\documentclass{amsproc}

\usepackage{verbatim} 
\usepackage{amsmath}
\usepackage{amsfonts}
\usepackage{amssymb}
\usepackage{mathrsfs}
\usepackage{amsthm}
\usepackage{newlfont}

\usepackage{fullpage,amsmath, amssymb,amscd}
\usepackage{latexsym, epsfig, color}
\usepackage[mathscr]{eucal}
\usepackage{xypic, graphicx}
\usepackage{amscd}
\usepackage[all, cmtip]{xy}
\usepackage{tikz,tikz-cd,xcolor}
\usepackage{tikz-cd}
\usepackage{mathtools}

\usepackage{float, hyperref, nccmath, multirow, caption, subcaption}

\newcommand\cyr{%
 \renewcommand\rmdefault{wncyr}%
 \renewcommand\sfdefault{wncyss}%
 \renewcommand\encodingdefault{OT2}%
\normalfont\selectfont} \DeclareTextFontCommand{\textcyr}{\cyr}

\newtheorem{theorem}{Theorem}
\newtheorem{lemma}[theorem]{Lemma}
\newtheorem{corollary}[theorem]{Corollary}
\newtheorem{proposition}[theorem]{Proposition}

\def\Z{\mathbb Z}
\def\N{\mathbb N}
\def\Q{\mathbb Q}
\def\R{\mathbb R}
\def\C{\mathbb C}
\def\F{\mathbb F}

\def\P{\mathbb P}

\def\G{\mathbb G}

\def\fp{\mathfrak p}

\def\Z{\mathbb Z}
\def\N{\mathbb N}
\def\Q{\mathbb Q}
\def\R{\mathbb R}
\def\C{\mathbb C}
\def\F{\mathbb F}

\def\P{\mathbb P}

\def\G{\mathbb G}

\def\T{\mathbb T}

\def\fp{\mathfrak p}

\def\Im{\operatorname{Im}}
\def\Re{\operatorname{Re}}

\def\det{\operatorname{det}}
\def\diag{\operatorname{diag}}

\def\tr{\operatorname{tr}}

\def\End{\operatorname{End}}

\def\Aut{\operatorname{Aut}}
\def\Gal{\operatorname{Gal}}

\def\Frob{\operatorname{Frob}}

\def\mod{\operatorname{mod}}

\def\disc{\operatorname{disc}}
\def\car{\operatorname{char}}

\def\gcd{\operatorname{gcd}}

\def\GL{\operatorname{GL}}
\def\GSp{\operatorname{GSp}}

\def\PGL{\operatorname{PGL}}

\def\Li{\operatorname{Li}}

\def\O{\operatorname{O}}

\def\log{\operatorname{log}}

\def\ds{\displaystyle}

\newcommand\numberthis{\addtocounter{equation}{1}\tag{\theequation}}

\begin{document}

\title{
Distribution of primes of split reductions for abelian surfaces
}

\date{}

\author[Wang]{Tian Wang}
\address{Tian Wang, Max Planck Institute for Mathematics, Vivatsgasse 7, Bonn, Germany, 53111}
\email{twang@mpim-bonn.mpg.de; \ twang213@uic.edu}

\renewcommand{\thefootnote}{\fnsymbol{footnote}}
\renewcommand{\thefootnote}{\arabic{footnote}}

\renewcommand{\thefootnote}{\fnsymbol{footnote}}
\renewcommand{\thefootnote}{\arabic{footnote}}

\thanks{}


\begin{abstract}
Let $A$ be an absolutely simple abelian surface defined over a number field $K$ with a commutative (geometric) endomorphism ring. 
Let $\pi_{A, \text{split}}(x)$ denote the number of primes $\fp$ in $K$ such that each prime has norm bounded by $x$, of good reduction for $A$, and the reduction of $A$ at $\fp$ splits. It is known that the density of such primes is zero.  
Under the Generalized Riemann Hypothesis for Dedekind zeta functions and possibly extending the field $K$, we prove that $\pi_{A, \text{split}}(x) \ll_{A, K} x^{\frac{41}{42}}\log x$ if the endomorphism ring of $A$ is trivial; $\pi_{A, \text{split}}(x) \ll_{A, F, K} \frac{x^{\frac{11}{12}}}{(\log x)^{\frac{2}{3}}}$ if $A$ has real multiplication by a real quadratic field $F$; $\pi_{A, \text{split}}(x) \ll_{A, F, K} x^{\frac{2}{3}}(\log x)^{\frac{1}{3}}$ if $A$ has complex multiplication by a CM field $F$. These results improve the bounds by  J. Achter in 2012 and D. Zywina in 2014. We also provide  better bounds under other credible conjectures. 
\end{abstract}

\maketitle
\section{Introduction}
Let $A$ be an absolutely simple abelian variety of dimension $g$ defined over a number field $K$.  Let $N_A$ be the norm of the conductor ideal of $A$. For any prime $\fp\nmid N_A$ (i.e., $\fp$ is coprime to $N_A$) that is in the set of nonzero prime ideals  $\Sigma_K$ of $K$, the reduction of $A$ modulo $\fp$, denoted by $\overline{A}_{\fp}$, is  an abelian variety defined over $\F
_{\fp}$.  A natural question to ask is if  $\overline{A}_{\fp}$ splits, i.e., if $\overline{A}_{\fp}$ is isogenous over $\F_{\fp}$ to a product of  smaller dimensional abelian subvarieties over $\F_\fp$.

It turns out that there is a strong connection between the (upper) density of primes $\fp$ such that $\overline{A}_{\fp}$ splits and the structure of the geometric endomorphism ring $\End_{\overline{K}}(A)$ of $A$. It is known that for an absolutely simple abelian variety $A/K$,  if $\End_{\overline{K}}(A)\simeq \Z$ and the dimension of $A$ is 2, 6 or odd,  the density of primes $\fp$ such that $\overline{A}_{\fp}$ splits is zero (see \cite[Corollary 6.10, p. 179] {Ch97});  if $A$ has complex multiplication or is of $\GL_2$-type over $\Q$ such that  $\End_{\overline{K}}(A)$ commutative,  a similar result holds (see \cite[Theorem 3.1, p. 12, Theorem 4.1, p. 17]{MuPa08}). On the other hand, if $\End_{\overline{K}}(A)$ is isomorphic to an indefinite quaternion division algebra, then for each prime $\fp\nmid N_A$,  $\overline{A}_\fp$ splits (see \cite[Theorem 2, p. 140]{Ta66}). In light of these observations, Murty and Patankar  \cite[Theorem 5.1, p. 3]{MuPa08} conjectured that the density of primes $\fp$ such that $\overline{A}_{\fp}$ splits is zero if and only if $\End_{\overline{K}}(A)$ is commutative.

In 2009, Achter proved one direction of the Murty-Patankar Conjecture  \cite[Theorem B]{Ac09}, and he also showed the conjecture holds for abelian varieties  that satisfy the Mumford-Tate Conjecture  \cite[Theorem A]{Ac09}. Building on these results, Zywina \cite[Theorem 1.2, p. 5043]{Zy14} proved the other direction of the Murty-Patankar Conjecture  under the Mumford-Tate Conjecture. 

In the case of an absolutely simple abelian surface $A/K$, more information is available due to the validity of the Mumford-Tate Conjecture (see e.g., \cite[Corollary 4.5, p.12]{Lo16}).  From the discussions above,  if $\End_{\overline{K}}(A)$ is commutative, then for almost all primes $\fp$, $\overline{A}_{\fp}$ does not split into a product of two elliptic curves. Therefore, it is natural to investigate the exceptional primes, namely, the primes for which $\overline{A}_{\fp}$ splits.

Motivated by the Lang-Trotter philosophy \cite{LaTr76}, Achter and Howe \cite[Conjecture 1.1, p. 40]{AcHo17} predicted that for a principally polarized and absolutely simple abelian surface  $A/K$ with $\End_{\overline{K}}(A)\simeq \Z$,  there exists a constant $C_{A/K}>0$ depending  on $A/K$, such that as $x\to \infty$,
\begin{equation}\label{AH-conj}
\pi_{A, \text{split}}(x)\coloneqq \#\{\fp \in \Sigma_K: N(\fp)\leq x, \fp\nmid N_{A}, \overline{A}_{\fp} \text{ splits}\}  \sim C_{A/K}\frac{x^{\frac{1}{2}}}{\log x}.
\end{equation}

Furthermore, upper bounds for $\pi_{A, \text{split}}(x)$ when $A$ is an absolutely simple abelian variety are given by Achter \cite[Theorem B, pp. 42-43]{Ac12} and Zywina \cite[Theorem 9.2, p. 5076]{Zy14}. They related the property that  $\overline{A}_{\fp}$ splits with the property that  the characteristic polynomial of  the Frobenius endomorphism  of $\overline{A}_{\fp}$  is reducible. By using  large sieve for Galois representations, they proved that for an absolutely simple abelian variety $A/K$ of dimension $g$ such that $\End_{\overline{K}}(A)$ is commutative, assuming the Generalized Riemann Hypothesis for Dedekind zeta functions (GRH) and the Mumford-Tate Conjecture for $A$  and by possibly increasing the field $K$, the following bound holds: 
\begin{equation}\label{M-T-bound}
\pi_{A, \text{split}}(x) \ll_{A, K} x^{1-\frac{1}{4d+2r+2}}(\log x)^{\frac{2}{2d+r+1}},
\end{equation}
where $d$ and $r$ are the dimension and the rank of the Mumford-Tate group of $A$, respectively.  Murty \cite[p. 267]{Mu05} also stated that if $A/\Q$ is absolutely simple of $\GL_2$-type with a commutative geometric endomorphism ring, then $\pi_{A, \text{split}}(x) \ll_{A, K} \frac{x^{\frac{13}{14}}}{(\log x)^{\frac{1}{7}}}$. More recently,  Shankar and Tang \cite{ShTa20} 
 proved that, though the density of exceptional primes is zero, there are infinitely many such primes. 

In order to specialize the result of (\ref{M-T-bound}) to an abelian surface, we need to understand the endomorphism algebra of $A$.  By Albert's classification of endomorphism algebras \cite[p. 203]{Mu70}, if $A$ is an abelian surface, then $\End_{\overline{K}}(A)$ 
 is either trivial or isomorphic to an order in a real quadratic field or an order in a CM quartic field. Then by (\ref{M-T-bound}), we obtain the following results in each case:
 \begin{enumerate}
 \item[(a)] if $\End_{\overline{K}}(A)\simeq \Z$, then $\pi_{A, \text{split}}(x) \ll_{A, K} x^{\frac{51}{52}} (\log x)^{\frac{1}{13}}$; \label{AcZy}
 \item[(b)]  if $\End_{\overline{K}}(A)$ is isomorphic to an order in a real quadratic field, then $\pi_{A, \text{split}}(x)  \ll_{A, K} x^{\frac{35}{36}} (\log x)^{\frac{1}{9}}$; \label{AcZy-RM}
 \item[(c)]  if $\End_{\overline{K}}(A)$ is isomorphic to an order in a CM quartic field, then $\pi_{A, \text{split}}(x) \ll_{A, K} x^{\frac{19}{20}} (\log x)^{\frac{1}{5}}$. \label{AcZy-CM}
 \end{enumerate}

By providing a detailed description of exceptional primes implied by a criterion of Maisner and Nart, along with a counting method of Cojocaru and David, we obtain upper bounds for $\pi_{A, \text{split}}(x)$ that improve upon (a), (b), and (c) under GRH and give better  bounds under additional credible hypotheses.


\begin{theorem}\label{main-thm}
Let $A/K$ be an absolutely simple abelian surface with $\End_{\overline{K}}(A)\simeq \Z$. We have that for all sufficiently large real number $x$, 
\begin{enumerate}
\item[(i)] assuming GRH, 
\[
\pi_{A, \text{split}}(x) \ll_{A, K} x^{\frac{41}{42}}(\log x)^{\frac{1}{21}};
\]
\item[(ii)] assuming GRH  and Artin's Holomorphic Conjecture (AHC) for Galois extensions of number fields, 
\[
\pi_{A, \text{split}}(x) \ll_{A, K} x^{\frac{21}{22}}(\log x)^{\frac{1}{11}};
\]
\item[(iii)] assuming GRH, AHC  and Pair Correlation Conjecture (PCC) for Galois extensions of number fields, 
\[
\pi_{A, \text{split}}(x) \ll_{A, K} x^{\frac{5}{6}}(\log x)^{\frac{1}{3}}.
\]
\end{enumerate}
\end{theorem}

\begin{theorem}\label{main-thm-RM}
Let $A/K$ be an absolutely simple abelian surface such that  $\End_{\overline{K}}(A)=\End_{K}(A)$ is isomorphic to an order in a real quadratic field $F$.  We have that for all sufficiently large real number $x$, 
\begin{enumerate}
\item[(i)] assuming GRH, 
\[
\pi_{A, \text{split}}(x) \ll_{A, F, K} \frac{x^{\frac{11}{12}}}{(\log x)^{\frac{2}{3}}};
\]
\item[(ii)] assuming GRH and AHC,
\[
\pi_{A, \text{split}}(x) \ll_{A, F, K}  \frac{x^{\frac{6}{7}}}{(\log x)^{\frac{3}{7}}};
\]
\item[(iii)] assuming GRH, AHC,  and PCC,
\[
\pi_{A, \text{split}}(x) \ll_{A, F, K}  x^{\frac{2}{3}}\log x.
\]
\end{enumerate}
\end{theorem}

\begin{theorem}\label{main-thm-CM}
Let $A/K$ be an absolutely simple abelian surface such that  $\End_{\overline{K}}(A)=\End_{K}(A)$ is isomorphic to an order in a  CM field $F$. Assume that $F\subseteq K$ and 
 that the  
kernel of the homomorphism $\phi: T_{F'}\to T_F$ defined in \cite[p. 85]{Ri80}  is connected, where $F'$, $T_F$, and $T_{F'}$ are the reflex field of $F$, the  algebraic tori defined by $F$ and by $F'$, respectively.
 Assuming GRH, we have that for all sufficiently large real number $x$,
\[
\pi_{A, \text{split}}(x) \ll_{A, F, K} x^{\frac{2}{3}}(\log x)^{\frac{1}{3}}.
\]
\end{theorem}



We make a few remarks on these theorems. Firstly, the bounds presented in Theorem \ref{main-thm} actually hold for the following counting function:
\begin{equation}\label{abs-simple}
\pi'_{A, \text{ split}}(x) :=\#\{\fp \in \Sigma_K: N(\fp) \leq x, \fp\nmid N_{A}, \overline{A}_{\fp} \text{ splits over } \overline{\F}_{\fp}\}.
\end{equation} 
We also do not need to assume the abelian surface $A/K$ is \emph{absolutely} simple. Indeed, in the proof, we only use the fact that $A/K$ is simple to 
get the desired $\ell$-adic Galois image  for all sufficiently large prime $\ell$.
Next, the assumptions made on $\End_{\overline{K}}(A)$ and $\phi$ in Theorem \ref{main-thm-RM} and Theorem \ref{main-thm-CM} are necessary to obtain an explicit description of the $\ell$-adic Galois image of $A$ for all sufficiently large primes $\ell$.
In Theorem \ref{main-thm-CM}, we do not consider the bounds under the assumption of AHC and PCC because, using our approach, the bound remains the same even when these conjectures are assumed. 
Lastly, Theorem \ref{main-thm-CM}
is nontrivial only if $K\neq \Q$, since it is known that $\End_{\overline{\Q}}(A)\neq \End_{\Q}(A)$ for abelian varieties $A$ over the rational $\Q$. 

 The proofs of Theorem 1, Theorem 2, and Theorem 3 are presented in Subsection 5.2, Subsection 5.3, and Subsection 5.4, respectively.  The proofs of these theorems rely on the characterization of an absolutely simple abelian surface defined over a finite field $\F_q$ using its characteristic polynomial of the Frobenius endomorphism, as discussed in Section 4. Then, in Section 5, we encode the prescribed properties of the characteristic polynomials into Chebotarev counting functions in   carefully chosen number field extensions associated to $A$, enabling us to successfully apply various versions of the effective Chebotarev Density Theorem.

More specifically, when $\End_{\overline{K}}(A)\simeq \Z$, we apply a square sieve, which is naturally applicable due to an evident square condition for the invariants of the prescribed  characteristic polynomial of  Frobenius. This is also one of the remarkable difference compared with the large sieve approach by Achter and Zywina. When $\End_{\overline{K}}(A)=\End_{K}(A)$ is isomorphic to an order in a real quadratic or a CM quartic field $F$, we take  advantage of the extra endomorphisms and prove that if $\fp$ is a prime such that $\overline{A}_{\fp}$ splits, then the corresponding characteristic polynomial of the Frobenius is a square, which offers more precision than the reducibility condition used by Achter and Zywina. This precision enables us to derive better upper bounds.

As an application of our results, we consider an  absolutely simple abelian surface $A/K$ and look at if the reduction of $A$ at $\fp\in \Sigma_K$ is isomorphic to the Jacobian of a genus 2 curve. There is not obvious because of the results in \cite[Theorem 1.2, p. 352]{Ru90}.
However, we prove that the reduction of $A$ at almost all primes is isogenous over  $\F_{\fp}$ to the Jacobian of a genus two (smooth projective) curve. Specifically, for an absolutely simple abelian surface $A/K$, we have for all sufficiently large real number $x$, 
 \begin{equation}\label{cor-1}
 \#\{\fp \in \Sigma_K: N(\fp)\leq x, \fp \nmid N_A, \overline{A}_{\fp} \text{ is not $\F_{\fp}$ isogenous to the Jacobian of a genus 2 curve}\}\ll \pi'_{A, \text{split}}(x).
 \end{equation}
In particular, by Theorem \ref{main-thm},  the number of primes $\fp\in \Sigma_K$ with $N(\fp)\leq x$ for which $\overline{A}_{\fp}$ is $\F_{\fp}$ isogenous to the Jacobian of a genus two curve, shares the same upper bound with $\pi_{A, \text{split}}(x)$  in Theorem \ref{main-thm}, under each assumption. 

As another application, for an absolutely simple abelian surface $A/K$ such that $\End_{\overline{K}}(A)$ is commutative,  we investigate the distribution of primes $\fp\in \Sigma_K$, $\fp\nmid N_A$ such that the number of $\F_{\fp}$-points of $\overline{A}_{\fp}$ reaches the extremal values $(q+1\pm \lfloor 2\sqrt{q}\rfloor)^2$, where $q=|\F_{\fp}|$ and  $\lfloor x \rfloor$ denotes the largest integer less than $x$. Investigation of primes with a similar property was first made by Serre \cite{Se83} for algebraic curves over finite fields. Recall that in the setting of elliptic curves over the rationals, such primes are called {\emph{extremal primes}}.  
We derive an upper bound for the number of extremal primes for abelian surfaces using our main theorems and a result from \cite{AuHaLa13}. The statement is the following.
\begin{corollary}\label{main-cor}
Let $A/K$ be an absolutely simple abelian surface. Then, we have  
 \[\#\{\fp \in \Sigma_K: N(\fp)\leq x, \fp\nmid N_A, |\overline{A}_{\fp}(\F_{\fp})|=(q+1\pm \lfloor 2\sqrt{q}\rfloor)^2 \}\ll \pi_{A, \text{split}}(x),\]
 where $q=|\F_{\fp}|$.
In particular,  under each setting and assumption  in  Theorem \ref{main-thm}, Theorem \ref{main-thm-RM}, and Theorem \ref{main-thm-CM}, the same bound holds for the number of extremal primes. 
\end{corollary}

In the last section, we expand on the split reductions for absolutely simple abelian surfaces of type II and IV according to Albert's classification and discuss the split reduction problem from a computational perspective. 
\bigskip

{\bf{Notation}}
\begin{itemize}
\item Let $\cal{A}$ be a finite set. We use $|\cal{A}|$ and $\#\cal{A}$ to denote the cardinality of $\cal{A}$, depending on whether the elements of $\cal{A}$ are spelled out or not. 
\item Given two real functions $f_1, f_2$,
we write $f_1 \ll f_2$
(or $f_1 = \O(f_2)$)
if 
 there exist positive constants $C$ and $x_0$ such that 
$|f_1(x)| \leq C \ |f_2(x)|$ for all $x>x_0$.
We denote by
$f_1 \ll_D f_2$ (or $f_1 = \O_D(f_2)$),
if 
$f_1 \ll f_2$
and
the  $\O$-constant $C$ depends on the property $D$.
We denote by $f_1 \asymp f_2$ if  $f_1 \ll f_2$ and $f_2\ll f_1$.
Finally, we say 
$f_1 \sim f_2$ 
if
$\lim_{x \rightarrow \infty} \frac{f_1(x)}{f_2(x)} = 1$.


\item We write  $p$ and $\ell$ for rational primes. 
We denote by  $\pi(x)$  the number of primes $p$ such that $p$ is no larger than a positive real number $x>2$. Recall 
the Prime Number Theorem asserts that
\[ \pi(x) \sim \Li (x) \sim \frac{x}{\log x},\]
where
$\Li (x)\coloneqq \int_{2}^x \frac{1}{\log t} \ d t$.

\item For any integer $n$,  let $v_p(n)$ denote the $p$-adic valuation of $n$. More specifically, if there is a decomposition $n=p^am$, where $a$ is a nonnegative integer and if  $\gcd(p, m)=1$, we have $v_p(n)=a$. 
If $n$ is a positive integer, we denote by $\Z/n\Z$ the ring of congruence classes modulo $n$. We denote by  $\Z_{p} \coloneqq \varprojlim_{n} \Z/p^n\Z$ as the ring of $p$-adic integers. For a prime power $q$, we denote by $\F_q$ the finite field with $q$ elements. 

\item For a fixed integer $a$, we use the standard notation  
$\left(\frac{a}{p} \right)$ for the Legendre symbol.  If $p\mid n$ and $a\in \Z/n\Z$, we use the notation $ \left(\frac{a}{p} \right)$ to denote  $\left(\frac{\hat{a}}{p} \right)$, where $\hat{a}\in \Z$ is any preimage of  
$a$ under the canonical projection
$\Z \twoheadrightarrow \Z/n\Z$.

\item 
Given a nonzero unitary commutative ring $R$, we denote by $R^{\times}$ the group of multiplicative units of $R$. 
For an integer $n \geq 1$, 
we use $M_{n\times n}(R)$ to denote the collection of $n \times n$ matrices with entries in $R$. We denote by $I_n\in M_{n\times n}(R)$ be the identity matrix and by $\diag(a_1, \ldots, a_n)\in M_{n\times n}(R)$
 the diagonal matrix with diagonal entries $a_1, \ldots, a_n\in R$.
For any  $M\in M_{n\times n}(R)$,  we denote by $M^{t}$, $\tr M$, $\det M$, and  $\car_M(X) \coloneqq \det(I-XM) \in R[X]$  the transpose of $M$, the trace of $M$,  the determinant of $M$, and the characteristic polynomial of $M$, respectively. 
We use the standard notation $\GL_{n}(R)$ for the general linear group.

\item  Let $K$ be a number field and let $\mathcal{O}_K$ be the ring of integers of $K$. We denote the absolute degree  of $K$ by  $n_K$ and the absolute discriminants of $K$ by  $d_K$. For each ideal $I\unlhd \cal{O}_K$, we denote  the  norm of the ideal $I$ by $N(I)$. We denote by $\Sigma_K$ the set of nonzero primes in $\cal{O}_K$. We say $\fp\in \Sigma_K$ has degree 1 if its norm  is a prime number, i.e., $N(\fp)=p$.

\item  We say that $K$ satisfies Generalized Riemann Hypothesis (GRH) if
 for any complex value $\rho$ in the critical strip $0 < \Re (\rho) < 1$ and is a zero of the Dedekind zeta function $\zeta_K(s)$ of $K$, 
  we have $\Re(\rho) = \frac{1}{2}$. 

\item We write $L/K$ to represent a finite extension of number fields.
 We write $\disc(L/K)\unlhd \mathcal{O}_{K}$ to denote the relative discriminant ideal of $L/K$.
 If $L/K$ is Galois, we use  $\Gal(L/K)$ to denote its Galois group. For each normal subgroup $N$ of $\Gal(L/K)$, we write $L^N$ to denote the fixed field of $L$ by $N$. By Galois theory, $L^N/K$ is a Galois extension with 
 $
 \Gal(L^N/K)\simeq \Gal(L/K)/N.
 $ 
We say the Artin's Holomorphic Conjecture (AHC) holds for the Galois extension $L/K$ if all the Artin L-functions attached to
nontrivial irreducible characters of $\Gal(L/K)$ are holomorphic on the complex plane.

\item Let $B$ be an abelian variety defined over $\F_q$. We say $B$ splits  (or $B$ is not simple)  if $B$ is isogenous over $\F_q$ to the product of smaller dimensional  abelian subvarieties defined over $\F_q$. We say  $B$ splits over $\overline{\F}_q$ (or $B$ is not absolutely simple) if  $B$ is isogenous over $\overline{\F}_q$ to a product of smaller dimensional abelian subvarieties.
\end{itemize}

\medskip
\noindent
{\bf{Acknowledgments.}} I want to thank Alina Carmen Cojocaru for introducing the Murty-Patankar Conjecture to me. I am thankful to  Manjul Bhargava for hosting me at Princeton University where I  carried out part of this work and to Yunqing Tang for the fruitful discussions while I was there. Furthermore, I want to acknowledge the valuable suggestions from Auden Hinz, Jacob Mayle, Pieter Moree, Ananth Shankar, Tonghai Yang, and Don Zagier. I wish to thank the Max-Planck-Institut f\"{u}r Mathematik in Bonn for its support and  inspiring atmosphere.
\bigskip

\section{Effective Chebotarev Density Theorem}
Let $L/K$ be a Galois extension of number fields with Galois group $G$. 
Let  $\mathcal{C}$ be a nonempty set in $G$ that is invariant under conjugation. For each unramified prime $\mathfrak{p}\in \Sigma_K$, we write  $\ds\left(\frac{L/K}{\mathfrak{p}} \right)\subseteq G$ for  the Artin symbol of $L/K$ at $\mathfrak{p}$. We set
\[
P(L/K) \coloneqq \{p \text{ rational prime}: \exists  \ \mathfrak{p}\in \Sigma_K, \text{ such that } \mathfrak{p}\mid p, \text{ and } \mathfrak{p}\mid \disc(L/K)\},
\]
\[
M(L/K) \coloneqq [L: K]d_K^{\frac{1}{n_K}}\prod_{p\in P(L/K)}p.
\]
The well-known Chebotarev Density Theorem states that for all sufficiently large real number $x$, 
\[
\pi_{\cal{C}}(x, L/K) \coloneqq \#\{\mathfrak{p}\in \Sigma_K:  \mathfrak{p} \text{ unramified, } N(\mathfrak{p})\leq x, \left(\frac{L/K}{\mathfrak{p}} \right)\subseteq \cal{C}\}\sim \frac{|\mathcal{C}|}{|G|}\Li(x).
\]
An effective version of the theorem with an explicit error term was proved by Lagarias and Odlyzko \cite{LaOd77} in the 1970s, which was then reformulated by Serre in the 1980s \cite{Se81}. Later on, Murty-Murty-Saradha  \cite{MuMuSa88} derived a better error term under GRH and  AHC. Afterwards, under the extra Pair Correlation Conjecture  for the Artin L-functions attached to irreducible characters of $G$ (see Section 3 of \cite{MuMuWo18}), M.R.  Murty and  V.K. Murty obtained a better error term of  $\pi_{\cal{C}}(x, L/K)$  which was refined by Murty-Murty-Wong \cite{MuMuWo18}. The following theorem summarizes the aforementioned results.

\begin{theorem}\label{effective-CDT}
Let $L/K$ be a finite Galois extension of number fields with Galois group $G$ and let $\mathcal{C}$ be a subset of $G$ that is  invariant under conjugation.  We have that for all sufficiently large real number $x$, 
\begin{enumerate}
\item[(i)] assuming GRH for the Dedekind zeta function of $L$, 
\[
\pi_{\cal{C}}(x, L/K) =\frac{|\mathcal{C}|}{|G|}\Li(x)+\O\left(\frac{|\mathcal{C}|}{|G|}x^{\frac{1}{2}}\left(\log |d_L|+n_L\log x\right)\right);
\]  \label{GRH}
\item[(ii)] assuming GRH for the Dedekind zeta function of $L$ and AHC for Artin L-functions associated to $L/K$,
\[
\pi_{\cal{C}}(x, L/K) =\frac{|\mathcal{C}|}{|G|}\Li(x)+\O\left(|\mathcal{C}|^{\frac{1}{2}}x^{\frac{1}{2}}n_K\log \left(M(L/K)x\right)\right);
\]  \label{GRH-AHC}
\item[(iii)] assuming GRH for the Dedekind zeta function of $L$, AHC and PCC for Artin L-functions associated to $L/K$,
\[
\pi_{\cal{C}}(x, L/K) =\frac{|\mathcal{C}|}{|G|}\Li(x)+\O\left(|\mathcal{C}|^{\frac{1}{2}}n_K^{\frac{1}{2}}\left(\frac{|G^{\#}|}{|G|} \right)^{\frac{1}{2}}x^{\frac{1}{2}}\log \left(M(L/K)x\right)\right),
\] 
where $G^{\#}$ denotes the set of conjugacy classes of $G$. \label{GRH-AHC-PCC}
\end{enumerate}
\end{theorem}
\begin{proof}
Statement (i) is stated in \cite[Th\'eor\`eme~4, p.~133]{Se81},  (ii) is a reformulation of \cite[Corollary 3.7, p. 265]{MuMuSa88}, and (iii) is stated in \cite[Theorem 1.2, p. 402]{MuMuWo18}. 
\end{proof}
 
We also record the following lemma due to Hensel on the upper bound of the discriminant of a finite number field extension. 

\begin{lemma}\label{hensel}
Under the same setting and notation as above, we have
\[
\log \left|N (\disc (L/K))\right|
\leq
(n_L - n_K)
\left(
\ds\sum_{p \in {P}(L/K)} \log p
\right)
+
n_L \log [L : K]
\]
\end{lemma}
\begin{proof}
See \cite[Proposition 5, p.~129]{Se81}.
\end{proof}

\bigskip

\section{Counting subsets of finite general symplectic groups}
The main goal of this section is to introduce some distinguished sets related to finite general symplectic groups and count their cardinalities. The results will be used while applying the effective Chebotarev Density Theorem. 

\subsection{Background}
 For now, fix an arbitrary integer $g\geq 1$. Later, we will specialize to the case where $g=2$.  
 Let $R$ be a nonzero unitary commutative ring. We write
$
J \coloneqq \left(
\begin{matrix} 
0 & I_{g}\\
-I_g & 0
\end{matrix}\right)
$
and the general symplectic group
\[
\GSp_{2g}(R) \coloneqq \{M\in \GL_{2g}(R): \exists \mu(M)\in R^{\times}, \text{ such that } M^{t}JM=\mu(M) J\}.
\]
For each $M\in \GSp_{2g}(R)$, the element $\mu(M)\in R^{\times}$ is called the multiplier of $M$. Recall that taking multipliers in $\GSp_{2g}(R)$ gives rise to a multiplicative character.
Fix a positive integer $m$ and an element $\gamma\in (\Z/m\Z)^{\times}$, we write
\begin{align*}
G(m) & \coloneqq \GSp_{2g}(\Z/m\Z), \\ 
G^{\gamma}(m)& \coloneqq \{M\in G(m): \mu(M)=\gamma\}, \\
\Lambda(m) & \coloneqq\{\lambda I_{2g}: \lambda \in (\Z/m\Z)^{\times}\},\\
P(m) & \coloneqq G(m)/\Lambda(m),\\
T(m) & \coloneqq \{\diag(\lambda_1, \ldots, \lambda_g, \lambda'_1, \ldots, \lambda'_g): \prod_{1\leq i\leq g}\lambda_i = \prod_{1\leq i\leq g}\lambda_i', \ \lambda_i, \lambda_i'\in (\Z/m\Z)^{\times} \  \forall 1\leq i\leq g\}, \\
\G(m) &\coloneqq\{(M_1, M_2)\in \GL_g(\Z/m\Z)\times \GL_g(\Z/m\Z): \det M_1=\det M_2\},\\
\P(m) & \coloneqq \G(m)/\Lambda(m),\\
\T(m) & \coloneqq T(m)/\Lambda(m).
\end{align*}
We write $G(m)^{\#}$, $P(m)^{\#}$, $\G(m)^{\#}$ and $\P(m)^{\#}$  to denote the set of conjugacy classes of the groups $G(m)$, $P(m)$, $\G(m)$ and $\P(m)$, respectively. 

Recall that for  $M\in G^{\gamma}(m)$, the characteristic polynomial of $M$ is a reciprocal polynomial of the form 
\[
\car_M(X) =X^{2g}+a_1X^{2g-1}+\ldots +a_gX^g +\ldots +\gamma^{g-1}a_1X+\gamma^g \in (\Z/m\Z)[X].
\]
For an $N=(N_1, N_2)\in \G(m)$, the characteristic polynomial of $N$ is a reducible polynomial of the form
\[
\car_N(X)= (X^{g}+b_{1}X^{g-1}+\ldots +d)(X^{g}+c_{1}X^{g-1}+\ldots +d) \in (\Z/m\Z)[X],
\]
where  $d = \det N_1 =\det N_2 \in (\Z/m\Z)^{\times}$.

Conversely, given a reciprocal polynomial $f(X)\in (\Z/m\Z)[X]$, the following lemmas provide essential properties of $f(X)$ that will be employed in the proofs of Proposition \ref{countings} and Proposition \ref{countings-2} presented at the end of this section. To properly state these lemmas, we introduce the notation $f(m)\ll g(m)$ for any integer $m$  to indicate that  for arithmetic functions $f(n), g(n): \N \to \C$, there exists a constant $C$ independent of $m$ such that $f(m)\leq C g(m)$. Likewise,  we use the notation $f(m)\asymp g(m)$ and $f(m)=\O(g(m))$  for any integer $m$.



\begin{lemma}\label{matrix-poly-lem}
 Let $m$ be any positive squarefree  integer such that $2, 3\nmid m$.   For $\gamma\in (\Z/m\Z)^{\times}$, we set
 \[
 f(X)=X^{2g}+a_1X^{2g-1}+\ldots+a_gX^g+\ldots + \gamma^{g-1}a_1X+\gamma^g\in (\Z/m\Z)[X].
 \]
 Then, 
  \begin{equation}\label{matrix-poly-eq}
 \#\{M\in G^{\gamma}(m): \car_M(X)=f(X)\}\asymp m^{2g^2}.
\end{equation}
Furthermore,  if $m=\ell$ is a prime number, then
 \begin{equation}\label{matrix-poly-eq-prime}
 \#\{M\in G^{\gamma}(\ell): \car_M(X)=f(X)\}=\ell^{2g^2}+\O(\ell^{2g^2-1}).
\end{equation}
\end{lemma}
\begin{proof}
If $m=\ell$ is a prime number such that  $\ell\geq 5$, then the result  is already proved in \cite[Theorem 3.5, p. 159]{Ch97}. 

If $m$ is a squarefree positive integer, we write its prime decomposition as
$m= \prod_{1\leq i\leq k} \ell_i$.  By the Chinese Remainder Theorem, 
 we have an isomorphism  of groups 
$G(m)\simeq \prod_{1\leq i\leq k}G(\ell_i)$. For $M\in G(m)$ and  a fixed $1\leq i\leq k$, we write $M_i\in G(\ell_i)$ to denote the image of $M$ under the projection from $G(m)$ to $G(\ell_i)$. Then $M_i \equiv M (\mod \ell_i)$ and
$
 \car_{M_i}(X)\equiv \car_M(X) (\mod \ell_i).
$

Let $\gamma_i\in (\Z/\ell_i\Z)^{\times}$ be the unique element such that $\gamma\equiv \gamma_i(\mod \ell_i)$ and $\mu(M_i)=\gamma_i$ for all $1\leq i\leq k$.
Based on the observations above,  we obtain
\begin{align*}
\#\{M\in G^{\gamma}(m): \car_M(X)=f(X)\} & =\prod_{1\leq i\leq k}\#\{M_i\in G(\ell_i)^{\gamma_i}:  \car_{M_i}(X) \equiv f(X)(\mod \ell_i)\}\\
& = \prod_{1\leq i\leq k}(\ell_i^{2g^2}+\O(\ell_i^{2g^2-1}))\asymp 
m^{2g^2}.
\end{align*}
\end{proof}
\begin{lemma}\label{matrix-poly-lem-RM}
 Let $\ell\geq 5$ be a prime.  Then for any monic polynomial 
 $
 f(X) \in \F_\ell[X]
 $
 of degree $g$, 
  we have
  \begin{equation}\label{matrix-poly-eq-RM}
 \#\{N\in \GL_{g}(\Z/\ell \Z): \car_N(X)=f(X)\}= \ell^{g^2-g}+\O(\ell^{g^2-g-1}).
\end{equation}
\end{lemma}

\begin{proof}
This is a direct result  from \cite[Theorem 3.9, p. 162]{Ch97}. 
\end{proof}


\subsection{Counting results in $\GSp_4$}
Let $m\geq 1$ be an arbitrary positive integer. We keep the notation from the earlier section but focus on $g=2$.

Recall that for any $M\in G(m)$, the characteristic polynomial of $M$ is a reciprocal polynomial of the form
\[
\car_M(X)=X^4+a_1X^3+a_2X^2+\gamma a_1X+\gamma^2 \in (\Z/m\Z)[X],
\]
where $\gamma=\mu(M)$. We define  the {\emph{discriminant}} of $M$ as
\begin{equation}\label{disc-matrix}
\Delta_M\coloneqq a_1^2-4 a_2+8\gamma \in \Z/m\Z.
\end{equation}

The following lemma shows when a reciprocal polynomial of degree 4 is reducible.
\begin{lemma}\label{quad-symb}
Let $\ell$ be an odd prime and $\gamma\in (\Z/\ell\Z)^{\times}$.
Let 
\[
f(X)=X^4+a_1X^3+a_2X^2+\gamma a_1X+\gamma^2\in (\Z/\ell\Z)[X].
\]
We denote by  $\Delta_f\coloneqq a_1^2-4a_2+8\gamma.$
Then, we have
\begin{enumerate}
\item[(i)] $f(X)=(X^2+\beta_1X+\mu)(X^2+\beta_2X+\mu)$ in $(\Z/\ell\Z)[X]$ for $\beta_1\neq \beta_2$ if and only if $\left(\frac{\Delta_f}{\ell} \right)=1$.
\item[(ii)] $f(X)=(X^2+\beta X+\mu)^2$ in  $(\Z/\ell\Z)[X]$ if and only if $\left(\frac{\Delta_f}{\ell} \right)=0$.
\end{enumerate}

\end{lemma}
\begin{proof}
\begin{enumerate}
\item[(i)]
If $\left(\frac{\Delta_f}{\ell} \right)=1$, then by definition, there is an element $u \in (\Z/\ell\Z)^{\times}$  such that 
$
u^2 =a_1^2-4a_2+8\gamma.
$
We denote by  $2^{-1}$ the inverse of $2$ in $(\Z/\ell\Z)^{\times}$. Then a direct computation shows that
\[
f(X) = (X^2+2^{-1}(a_1+u)X+\gamma)(X^2+2^{-1}(a_1-u)X+\gamma) \in (\Z/\ell\Z)[X].
\] 
The statement follows by noting that $\ds 2^{-1}(a_1+u)\neq 2^{-1}(a_1-u)$.

Conversely, if  $f(X)=(X^2+\beta_1X+\gamma)(X^2+\beta_2X+\gamma)$ for $\beta_1\neq \beta_2\in \Z/\ell\Z$, then a direct computation shows  that $\Delta_f=(\beta_1-\beta_2)^2\in (\Z/\ell\Z)^{\times}$. Therefore, by definition, we get $\left(\frac{\Delta_f}{\ell} \right)=1$.
\item[(ii)] 
If $\left(\frac{\Delta_f}{\ell} \right)=0$,  then
$
a_1^2-4a_2+8\gamma=0.
$
In this case, one can check the following  factorization holds:
\[
f(X) = (X^2+2^{-1}a_1X+\gamma)^2 \in (\Z/\ell\Z)[X].
\] 

Conversely, if  $f(X)=(X^2+\beta X+\gamma)^2$, then a direct computation shows $\Delta_f=0$.  
Therefore,  we get $\left(\frac{\Delta_f}{\ell} \right)=0$.
\end{enumerate}
\end{proof}


Fix an odd prime $\ell$, we introduce some subsets of  $G(\ell)$ that are  invariant under conjugation by elements of $G(\ell)$. 
We denote by 
\begin{align}
& \cal{C}^0(\ell) \coloneqq \{M\in G(\ell): \Delta_M=0\}. \numberthis\label{conjugacy-0}
\end{align}
For $ -2 \leq i\leq 6$, we write  
\begin{equation}\label{conjugacy-1}
\cal{C}_i^1(\ell) \coloneqq \{M \in G(\ell): \car_M(X)=X^4+a_{1}X^3+a_{2}X^2+\gamma a_1X+\gamma^2\in (\Z/\ell\Z)[X], a_2= i\gamma\},
\end{equation}
We also write
\begin{align} 
&\cal{C}^2(\ell) \coloneqq \{M \in G(\ell): \car_M(X)=X^4+a_{1}X^3+a_{2}X^2+\gamma a_1X+\gamma^2\in (\Z/\ell\Z)[X], a_1=0\}, \numberthis \label{conjugacy-2}\\
& \cal{C}^3(\ell)  \coloneqq \{M \in G(\ell): \car_M(X)=X^4+a_{1}X^3+a_{2}X^2+\gamma a_1X+\gamma^2\in(\Z/\ell\Z)[X], a_1^2=\gamma+a_2\}, \numberthis \label{conjugacy-3}\\
&\cal{C}^4(\ell) \coloneqq \{M \in G(\ell): \car_M(X)=X^4+a_{1}X^3+a_{2}X^2+\gamma a_1X+\gamma^2\in (\Z/\ell\Z)[X], a_1^2=2a_2\}, \numberthis \label{conjugacy-4}\\
& \cal{C}^5(\ell)  \coloneqq \{M \in G(\ell): \car_M(X)=X^4+a_{1}X^3+a_{2}X^2+\gamma a_1X+\gamma^2\in (\Z/\ell\Z)[X], a_1^2=3(a_2-\gamma)\}. \numberthis \label{conjugacy-5}
\end{align}

We also consider subsets of $P(\ell)$ that are invariant under conjugation. The construction is based on the following observation.
For any $\lambda\in (\Z/\ell\Z)^{\times}$ and $M\in G^{\gamma}(\ell)$, we have
\[
 \car_{\lambda M}(X)=X^4+\lambda a_{1}X^3+\lambda^2 a_{2}X^2+(\gamma\lambda^2) \lambda a_1X+(\gamma\lambda^2)^2.
\]
Therefore, by definition (\ref{disc-matrix}), we have
\begin{equation}\label{quadratic-P-G}
\left(\frac{\Delta_{\lambda M}}{\ell} \right) = \left(\frac{\lambda^2\Delta_M}{\ell} \right)=\left(\frac{\Delta_M}{\ell} \right).
\end{equation}
 For any  $g\in P(\ell)$, we denote its lifting in $G(\ell)$ as $\hat{g}$. Note that any two distinct liftings of $g$ differ by the multiplication of a scalar matrix in  $\Lambda(\ell)$. Thus based on the observation (\ref{quadratic-P-G}), the following sets are well-defined:
\begin{align}
\cal{C}_1(\ell) & \coloneqq \{g\in P(\ell): \left(\frac{\Delta_{\hat{g}}}{\ell} \right)=1\}, \label{conj-P-1}\\
\cal{C}_0(\ell) & \coloneqq \{g\in P(\ell): \left(\frac{\Delta_{\hat{g}}}{\ell} \right)=0\},\label{conj-P-0}\\
\cal{C}_{-1}(\ell) & \coloneqq \{g\in P(\ell): \left(\frac{\Delta_{\hat{g}}}{\ell} \right)=-1\},\label{conj-P--1}
\end{align}

Similarly, for any $N=(N_1, N_2)\in \G(m)$, where $N_1, N_2\in \GL_2(\Z/m\Z)$, recall that the characteristic polynomial of $N$ can be written as
\begin{align*}
\car_N(X) & =\car_{N_1}(X)\car_{N_2}(X)\\
 &=(X^2+b_1X+d)(X^2+c_1X+d) \in (\Z/m\Z)[X],  d=\det N_1=\det N_2 \in (\Z/m\Z)^{\times}. 
\end{align*}
If $m=\ell$ is a prime, we define the following subsets of  $\P(\ell)$ that are invariant under conjugation. First observe that for any $\lambda\in (\Z/\ell\Z)^{\times}$ and any $N=(N_1, N_2)\in \G(\ell)$,  
\begin{equation}\label{conjugacy-2}
\tr \lambda N_i=\lambda \tr N_i, \ \ \det \lambda N_i =\lambda^2\det N_i, \ 1\leq i\leq 2.
\end{equation}
Now given $(g_1, g_2)\in \P(\ell)$, we write its lifting in $\G(\ell)$ as $(\hat{g}_1, \hat{g}_2)$. Note that any two distinct liftings of $g$ and $(g_1, g_2)$ differ by the multiplication  of a scalar matrix in  $\Lambda(\ell)$. Thus based on the observation (\ref{conjugacy-2}), the following unions of conjugacy classes are well defined:
\begin{align}
\cal{C}_1^{RM}(\ell)& \coloneqq \{(g_1, g_2)\in \P(\ell): \tr \hat{g}_1=\tr \hat{g}_2\},\label{conj-P-1-RM}\\
\cal{C}^{CM}(\ell) &\coloneqq \{(g_1, g_2)\in \T(\ell): \tr \hat{g}_1=\tr \hat{g}_2\}.\label{conj-P-1-CM}
\end{align}

Now we are ready to state the two counting propositions. 
 
\begin{proposition}\label{countings}
Let $m$ be any positive squarefree integer such that $2, 3\nmid m$. Let $\ell\geq 5$ be any prime.
\begin{enumerate}
\item[(i)] 
We have
\begin{align*}
 & | G(m)|  \asymp m^{11},  \; |P(m)| \asymp m^{10},  \\
 & |G(\ell)| =\ell^{11}+\O(\ell^{10}), \; \; \; |P(\ell)|   =\ell^{10}+\O(\ell^8). 
\end{align*} 
 \label{G-P}
\item[(ii)]
There exists a constant  $C_{\ell}$ that is uniformly bounded (independent of $\ell$) such that 
\begin{align*}
|\cal{C}_1(\ell)|& =\frac{1}{2}\ell^{10}+C_{\ell}\ell^{9}+\O(\ell^8), \\
 |\cal{C}_0(\ell)| &=\ell^{9}+\O(\ell^8), \\
  |\cal{C}_{-1}(\ell)|& =\frac{1}{2}\ell^{10}-(1+C_{\ell})\ell^{9}+\O(\ell^8),
  \end{align*}
\label{main-conj}
\item[(iii)] We have  $|G^{\#}(m)|\ll m^{3}$ and $|P^{\#}(m)|\ll m^{2}$. \label{G-P-conj}
\item[(iv)] We have $|\cal{C}_i^1(\ell)| \ll \ell^{10}$ for $ -2\leq i\leq 6$ and $|\cal{C}^j(\ell)|\ll \ell^{10}$ for $j=0$ and $2\leq j\leq 5$.\label{5-conj}
\end{enumerate} 
\end{proposition}

\begin{proof}
\begin{enumerate}
\item[(i)] The well-known formula $|\GSp_{2g}(\Z/\ell\Z)|=(\ell-1)\prod_{i=1}^g(\ell^{2i}-1)\ell^{2i-1}$  \cite[Theorem 3.1.2, p. 35]{Om78} yields 
\[
|G(\ell)|  = \ell^4(\ell-1)(\ell^2-1)(\ell^4-1) = \ell^{11}+\O(\ell^{10}).
\] 
It follows that
$
|P(\ell)|  = \ell^4(\ell^2-1)(\ell^4-1) = \ell^{10}+\O(\ell^8).
$

Since $m$ is a squarefree positive integer, we can write the prime decomposition of $m$ as 
$\displaystyle m=\prod_{1\leq i\leq k}\ell_i$. By the Chinese Remainder Theorem, we obtain
\begin{align*}
&|G(m)|  = \prod_{1\leq i\leq k}|G(\ell_i)|  \asymp \prod_{1\leq i\leq k} \ell_i^{11}=m^{11}, \\
 & |P(m)|=  \prod_{1\leq i\leq k}|P(\ell_i)|  \asymp \prod_{1\leq i\leq k} \ell_i^{10} = m^{10}.
\end{align*}
\item[(ii)] 
By Lemma \ref{quad-symb}, the value of $\left(\frac{\Delta_{\hat{g}}}{\ell} \right)$ reflects the decomposition type of the characteristic polynomial of $\hat{g}$ in $(\Z/\ell\Z)[X]$. Combining  with Lemma \ref{matrix-poly-lem}, we obtain 
\begin{align*}
|\cal{C}_1(\ell)|
& = \frac{1}{\ell-1} \sum_{\substack{\hat{g}\in G(\ell)\\\left(\frac{\Delta_{\hat{g}}}{\ell} \right)=1}}1 \\
& = \frac{1}{\ell-1}\sum_{\gamma\in (\Z/\ell\Z)^{\times}}\sum_{\substack{\hat{g}\in G^{\gamma}(\ell)\\\left(\frac{\Delta_{\hat{g}}}{\ell} \right)=1}}1\\
&= \frac{1}{\ell-1}\sum_{\gamma\in (\Z/\ell\Z)^{\times}}\sum_{\substack{f(X)\in (\Z/\ell\Z)[X]\\ f(X)=(X^2+\beta_1X+\gamma)(X^2+\beta_2X+\gamma)\\ \beta_1\neq \beta_2 \in \Z/\ell\Z}}\#\{\hat{g}\in G^{\gamma}(\ell): \car_{\hat{g}}(X)=f(X)\}\\
&=  \frac{1}{\ell-1}(\ell-1) \frac{\ell(\ell-1)}{2}(\ell^8+\O(\ell^7))\\
&=\frac{1}{2}\ell^{10}+C_{\ell}\ell^9+\O(\ell^8),
\end{align*}
where $C_{\ell}=\O(1)$. 

Similarly, we have
\begin{align*}
|\cal{C}_0(\ell)|
& = \frac{1}{\ell-1} \sum_{\substack{\hat{g}\in G(\ell)\\\left(\frac{\Delta_{\hat{g}}}{\ell} \right)=0}}1 \\
&=\frac{1}{\ell-1}\sum_{\gamma\in (\Z/\ell\Z)^{\times}}\sum_{\substack{\hat{g}\in G^{\gamma}(\ell)\\\left(\frac{\Delta_{\hat{g}}}{\ell} \right)=0}}1\\
&= \frac{1}{\ell-1}\sum_{\gamma\in (\Z/\ell\Z)^{\times}}\sum_{\substack{f(X)\in (\Z/\ell\Z)[X]\\ f(X)=(X^2+\beta X+\gamma)^2\\ \beta\in \Z/\ell\Z}}\#\{\hat{g}\in G^{\gamma}(\ell): \car_{\hat{g}}(X)=f(X)\}\\
&=  \frac{1}{\ell-1} (\ell-1) \ell(\ell^8+\O(\ell^7))\\
&=\ell^{9}+\O(\ell^8).
\end{align*}
From the observation 
\[
|P(\ell)|=|\cal{C}_1(\ell)|+|\cal{C}_0(\ell)|+|\cal{C}_{-1}(\ell)|
\]
and 
$ |P(\ell)|=\ell^{10}+\O(\ell^8)$ in part (i),  we obtain 
\[
|\cal{C}_{-1}(\ell)|=\frac{1}{2}\ell^{10}-(1+C_{\ell})\ell^9+\O(\ell^8).
\]
\item[(iii)] 
The upper bound for $|G^{\#}(\ell)|$ can be derived from \cite[(iii), p. 36]{Wa63} and \cite[(2), p. 175]{Ga70}. 
To obtain an upper bound for $|P(\ell)^{\#}|$, we proceed as follows.
Recall that each conjugacy class in $P(\ell)$ may be viewed as the equivalence class of $\Lambda(\ell)$-orbits of certain conjugacy class in $G^{\#}(\ell)$. 
Now fix an arbitrary element 
 $\cal{C}\in G^{\#}(\ell)$. If there is an element $b\in (\Z/{\ell}\Z)^{\times}$ such that $(bI) \cal{C}=\cal{C}$, 
 then, by comparing  determinants, we have
$b^{4}=1$. So $b$ may take at most $4$ elements.
 By the orbit-stabilizer theorem, each  $\Lambda(\ell)$-orbit 
 of $G^{\#}(\ell)$ contains at least 
 $\frac{| \Lambda(\ell)|}{4}$ conjugacy classes. Therefore, 
\[
|P({\ell})^{\#}| \leq  \frac{|G^{\#}(\ell)|}{\frac{|(\Z/\ell\Z)^{\times}|}{4}} \ll \ell^2.
\] 
 By the Chinese Remainder Theorem, we get
$|G^{\#}(m)|\ll m^{3}$ and $|P^{\#}(m)|\ll m^2$.

\item[(iv)] 
The sizes of the unions of conjugacy classes can be estimated almost identically, so we only present the proof for the bound of $|\cal{C}^0(\ell)|$. By Lemma \ref{matrix-poly-lem}, 
\begin{align*}
|\cal{C}^0(\ell)|  &=\sum_{a_1\in \Z/\ell\Z}\sum_{\gamma\in (\Z/\ell\Z)^{\times}} \#\{M\in G^{\gamma}(\ell): \car_M(X)=X^4+a_1X^3+a_2X^2+\gamma a_1X+\gamma^2, 4a_2=a_1^2+8\gamma\}\\
& \ll \sum_{a_1\in \Z/\ell\Z}\sum_{\gamma\in (\Z/\ell\Z)^{\times}} \ell^8= (\ell-1)\ell^{9}.
\end{align*}
\end{enumerate}
\end{proof}

\begin{proposition}\label{countings-2}
Let  $\ell\geq 5$ be any prime. We have 
\begin{enumerate}
\item[(i)] 
 $|\G(\ell)| =\ell^{7}+\O(\ell^6),   \ \ |\P(\ell)|  =\ell^{6}+\O(\ell^4).$
 \label{G-P-RM}
\item[(ii)]
$|\cal{C}^{RM}_{ 1}(\ell)| =\ell^{5}+\O(\ell^4).$
  \label{main-conj-RM}
\item[(iii)]   $|\G^{\#}(\ell)|\ll \ell^{3}$ and $|\P^{\#}(\ell)|\ll \ell^{2}$. \label{G-P-conj-RM}
\end{enumerate} 
\end{proposition}
\begin{proof}
\begin{enumerate}
\item[(i)] Note that by definition,
 \[
 \G(\ell)=\frac{|\GL_2(\Z/\ell\Z)|^2}{|(\Z/\ell\Z)^{\times}|}=(\ell - 1)^3 \ell^{2} (\ell +1)^{2}=\ell^7+\O(\ell^6).
 \]
  Thus, the first part can be deduced immediately.  The second part follows from the observation that
\[
|\P(\ell)|=\frac{|\G(\ell)|}{|\Lambda(\ell)|}=\frac{(\ell - 1)^3 \ell^{2} (\ell +1)^{2}}{\ell-1}=\ell^6+\O(\ell^4).
\]
\item[(ii)] By Lemma \ref{matrix-poly-lem-RM}, we have 
\begin{align*}
|\cal{C}_{ 1}^{RM}(\ell)| & =\frac{1}{\ell-1} \sum_{\substack{g=(\hat{g}_1, \hat{g}_2)\in \G(\ell)\\\tr \hat{g}_1= \tr \hat{g}_2}} 1\\
& = \frac{1}{\ell-1}\sum_{d\in (\Z/\ell\Z)^{\times}}\sum_{t\in \Z/\ell\Z} \#\{\hat{g}_1\in \GL_2(\Z/\ell\Z): \car_{\hat{g}_1}(X)=X^2+tX+d\} \\
& \hspace{100pt} \times \#\{\hat{g}_2\in \GL_2(\Z/\ell\Z): \car_{\hat{g}_2}(X)=X^2+ tX+d\}\\
& =\frac{1}{\ell-1} (\ell-1)\ell (\ell^2+\O(\ell))^2\\
& =\ell^5+\O(\ell^4).
\end{align*}


\item[(iii)] 
Recall from  \cite[Theorem, p. 91]{FeFi60} that $|\GL_2(\Z/\ell\Z)^{\#}|=\ell^2-1$.  Moreover,  each conjugacy class of $\GL_2(\Z/\ell\Z)^{\#}$ is generated by a matrix with one of the following types:
\begin{align*}
g_{1} = \left(\begin{matrix}a & 0 \\ 0 & b \end{matrix}\right), \quad & g_{2} = \left(\begin{matrix}a & 0 \\ 0 & a \end{matrix}\right), \\
g_{3} = \left(\begin{matrix}a & 1 \\ 0 & a \end{matrix}\right), \quad & g_{4} = \left(\begin{matrix}a & \epsilon b \\ b & a \end{matrix}\right),
\end{align*}
where $a\neq b$, $a, b\in (\Z/\ell\Z)^{\times}$ and $\epsilon$ is a fixed generator  in the multiplicative group $\F_{\ell^2}^{\times}$. Therefore,  the number of conjugacy classes with a fixed determinant in each type is at most $\ell-1$, $2$, $2$, $\ell-1$, respectively. So we have 
\[
|\G(\ell)^{\#}|\leq (\ell-1)(4+2(\ell-1))^2 \leq 4(\ell+1)^{2}(\ell-1).
\]
Fix an arbitrary element 
 $\cal{C}\in \G(\ell)^{\#}$. If there is an element $b\in (\Z/\ell\Z)^{\times}$ such that $(bI) \cal{C}=\cal{C}$, 
 then by comparing  determinants, we have
$b^{4}=1$. So $b$ take at most $4$ values in $(\Z/\ell\Z)^{\times}$.
 By the orbit-stabilizer theorem, each  $\Lambda(\ell)$-orbit 
 of $\G(\ell)^{\#}$ contains at least 
 $\frac{| \Lambda(\ell)|}{4}$ conjugacy classes. Therefore,
\[
|\P(\ell)^{\#}|\leq  \frac{|G(\ell)^{\#}|}{\frac{|(\Z/\ell\Z)^{\times}|}{4}} \leq  16(\ell+1)^{2}.
\]

\end{enumerate}
\end{proof}

\begin{proposition}\label{countings-CM}
Let  $\ell\geq 5$ be any prime. We have 

\begin{enumerate}
\item[(i)] 
 $|\T(\ell)| =\ell^{3}+\O(\ell^2).$
 \label{G-P-CM}
\item[(ii)]
$|\mathcal{C}^{CM}(\ell)| \ll \ell.$
  \label{main-conj-CM}
\end{enumerate} 
\end{proposition}
\begin{proof}
(i) Note that by definition, 
\[
|\T(\ell)|=\frac{|(\Z/\ell\Z)^{\times}|^{4}}{|(\Z/\ell\Z)^{\times}|}=(\ell-1)^3=\ell^{3}+\O(\ell^2).
\]

(ii) 
We denote by
\[
\cal{C}_1^{CM}(\ell) \coloneqq \{g=(g_1, g_2)\in T(\ell): \tr \hat{g}_1=\tr \hat{g}_2\}.
\]
Then, 
\begin{align*}
|\mathcal{C}_1^{CM}(\ell)| & =\sum_{\lambda_1, \lambda_2\in (\Z/\ell\Z)^{\times}} \#\{(\lambda_1', \lambda_2')\in \Z/\ell\Z\times \Z/\ell\Z:  \lambda'_1+\lambda'_2 = \lambda_1+ \lambda_2, \lambda_1' \lambda_2'= \lambda_1 \lambda_2\}\\ 
& \leq 2(\ell-1)^2 \ll \ell^2.
\end{align*}
Since $\mathcal{C}^{CM}(\ell)$ is the image of $\mathcal{C}_1^{CM}(\ell)$ in $\T(\ell)$, we obtain
\[
|\mathcal{C}^{CM}(\ell)|=\frac{|\mathcal{C}_1^{CM}(\ell)|}{|\Lambda(\ell)|}\ll \frac{\ell^2}{\ell-1} \ll \ell.
\]
\end{proof}

\bigskip

\section{Characterization of absolutely simple abelian varieties over finite fields}
\subsection{$q$-Weil polynomials}
Let $A/K$ be an abelian variety of dimension $g$, and let $N_A$ denote the norm of the conductor of $A$. For each positive integer $m$, we denote by $A[m]\subseteq A(\overline{K})$  the $m$-torsion subgroup of $A$, which is a free $(\Z/m\Z)$-module of rank $2g$. For each prime $\ell$, we denote by $T_{\ell}(A)\coloneqq\varprojlim_{n}A[\ell^n]$  the $\ell$-adic Tate module of $A$, which is a free $\Z_{\ell}$-module of rank $2g$. The actions of the absolute Galois group $\Gal(\overline{K}/K)$ on $A[m]$ and $T_{\ell}(A)$ induce the  mod-$m$ Galois representation and $\ell$-adic Galois representation, respectively, as follows,
\begin{align*}
\bar{\rho}_{A, m}: \Gal(\overline{K}/K) \to \Aut_{\Z/m\Z}(A[m]),\\
\rho_{A, \ell}: \Gal(\overline{K}/K) \to \Aut_{\Z_\ell}(T_{\ell}(A)).
\end{align*} 
We fix a $K$-polarization of $A$ of the minimal degree $d$ and let $\ell$ be a prime larger than $d$. The polarization pf $A$ and the existence of Weil pairing impose  constraints that imply $\Im \bar{\rho}_{A, \ell}\subseteq\GSp_{2g}(\Z/\ell\Z)$ and $\Im \rho_{A, \ell}\subseteq \GSp_{2g}(\Z_{\ell})$ (see \cite[pp. 32-33]{Lom16}).
Moreover, there exists a constant $C_K(A)$  depending only on $A$ and $K$, such that  for each integer $m$ with $(m, C_K(A))=1$, we have  $\overline{\rho}_{A, m}$ has  image equal to $\GSp_{2g}(\Z/m\Z)$.

For each  prime $\fp$ in $K$ with $\fp \nmid N_A$, we denote by  $\overline{A}_{\fp}$   the reduction of $A$ modulo $\fp$. We write $\Frob_{\fp}\subset \Gal(\overline{K}/K)$ to denote  the Frobenius conjugacy class at $\fp$. For any prime $\ell$ that is coprime to  $\fp$, by the N\'{e}ron-Ogg-Shafarevich criteria  for abelian varieties \cite{SeTa68},  the $\ell$-adic Galois representation is unramified at $\fp$. We denote the characteristic polynomial of $\rho_{A, \ell}(\Frob_{\fp})$ by
\[
P_{A, \fp}(X) = X^{2g}+ a_{1, \fp}(A)X^{2g-1} +\ldots +a_{g, \fp}(A)X^{g}+ \ldots +q^{g-1}a_{1, \fp}(A)X+q^g,
\]
where $q=N (\fp)$.
This  is a monic $q$-Weil polynomial in $\Z[X]$ which is independent of the choice of $\ell$.

Let $B$ be an abelian variety defined over $\F_q$ and let $\pi_{B}$ be its Frobenius endomorphism. Recall that there is a unique polynomial $P_{B}(X)\in \Z[X]$ which interpolates the degree of endomorphisms $n-\pi_B$ for all integers $n$ (see \cite[p. 57]{WaMi71}). For simplicity, we refer to this polynomial as \emph{the characteristic polynomial of $B$}. The following result  shows that the characteristic polynomial of an abelian variety defined over $\F_q$ determines its $\F_q$-isogeny class. 

\begin{theorem}\label{Honda-Tate}
Let $B_1, B_2$ be abelian varieties defined over $\F_q$. Then $B_1$ and $B_2$ are isogenous over $\F_q$ if and only if $P_{B_1}(X)=P_{B_2}(X)$.
\end{theorem}
\begin{proof}
See \cite[Theorem 1 (c), p. 139]{Ta66}.
\end{proof}


The following result illustrates the bounds of the coefficients of a $q$-Weil polynomial.  

\begin{lemma}\label{lem-2.1}
Let $q$ be a  prime power and let $f(X)=X^4+a_1X^3+a_2X^2+qa_1X+q^2\in \Z[X]$. Then $f(X)$ is a $q$-Weil polynomial if and only if
\[
|a_{1}|\leq 4\sqrt{q}, \ 2|a_{1}|\sqrt{q}-2q\leq a_2\leq \frac{a_1^2}{4}+2q.
\]
\end{lemma}
\begin{proof}
See \cite[Lemma 2.1, p. 323]{MaNa02}.
\end{proof}

\subsection{Absolutely simple abelian surfaces over finite fields}

Now we go back to the setting of abelian surfaces. The following result of Masiner and Nart \cite{MaNa02} provides a complete description of an absolutely simple abelian surface over a finite field via its characteristic polynomial.

\begin{lemma}\label{thm-2.15}
Let $f(X)=X^4+a_1X^3+a_2X^2+qa_1X+q^2\in \Z[X]$ be a $q$-Weil polynomial and let
\[
\Delta_f\coloneqq a_1^2-4a_2+8q,\  \delta_f\coloneqq (a_2+2q)^2-4qa_1^2.
\]
Then there exists an absolutely simple abelian surface $B$ defined over $\F_q$, where $q$ is a power of $p$, such that $P_{B}(X)=f(X)$ if and only if
\begin{enumerate}
\item[(i)] $\Delta_f$ is not a square in $\Z$, and
\item[(ii)] either  $v_p(a_2)=0$, $a_1^2\notin\{0, q+a_2, 2a_2, 3(a_2-q)\}$, or
 $v_p(a_1)=0$, $v_p(a_2)\geq \frac{\delta_f}{2}$ and  $\delta_f$ is not a square in $\Z_p$.
\end{enumerate} 
\end{lemma}
\begin{proof}
See \cite[Theorem 2.15, p. 326]{MaNa02}.
\end{proof}

Now we consider an abelian surface $A$ over a number field $K$. For a prime $\fp\in \Sigma_K$ such that $\fp\nmid N_A$, let 
\begin{equation}\label{char-poly-surface}
P_{\overline{A}_{\fp}}(X) = X^4+a_{1, \fp}X^3+a_{2, \fp}X^2+qa_{1, \fp}X+q^2\in \Z[X]
\end{equation}
be the characteristic polynomial of $\overline{A}_{\fp}/\F_{\fp}$, where $q=N(\fp)$.   It is well-known that 
\[
P_{\overline{A}_{\fp}}(X) =P_{A, \fp}(X).
\]
With this in mind,   from now on, we use the notation $P_{A, \fp}(X)$ instead of $P_{\overline{A}_{\fp}}(X)$.
We also  denote following  invariants of  $P_{A, \fp}(X)$:
\[
\Delta_{A, \fp}\coloneqq a_{1, \fp}^2-4a_{2, \fp}+8q, \ \  \delta_{A, \fp} \coloneqq (a_{2, \fp}+2q)^2-4q a_{1, \fp}^2.
\]

The first lemma is a general criterion for the primes $\fp\in \Sigma_K$ such that $\fp \nmid N_A$ for which the reduction of an abelian surface $A/K$ at $\fp$ is not absolutely simple. In particular, this applies to the case where $\End_{\overline{K}}(A)\simeq \Z$.
\begin{lemma}\label{main-lem}
Let $A/K$ be an abelian surface.
If there is a prime $\fp$ in $K$ of degree 1 such that  $\fp\nmid N_A$ for which $\overline{A}_{\fp}$ is not absolutely simple, then either $\Delta_{A, \fp}$ is a square in $\Z$ or one of the following is satisfied:
\begin{enumerate}
\item $a_{2, \fp}\in \{ip: -2\leq i\leq 6, i\in \Z\}$, \label{main-lem-1} 
\item $a_{1, \fp}^2\in \{0,p+a_{2, p}, 2a_{2, p}, 3(a_{2, p}-p)\}$, \label{main-lem-2} 
\end{enumerate}
where $p=N(\fp)$.
\end{lemma} 
\begin{proof}
As $\fp$ is a prime of degree 1, $\overline{A}_{\fp}$ is an abelian variety defined over the finite field $\F_{p}$. Fix a prime $\fp\nmid N_A$. If $\overline{A}_{\fp}$ is not absolutely simple, then it is not $\F_p$-isogenous to any absolutely simple abelian surfaces defined over $\F_p$. By Theorem \ref{Honda-Tate}, $P_{A, \fp}(X)$ is not the characteristic polynomial of any absolutely simple abelian surfaces over $\F_p$. So from Lemma \ref{thm-2.15}, either $\Delta_{A, p}$ is a square or $p$ satisfies  (\ref{main-lem-1}) or (\ref{main-lem-2}). 
Indeed, we can derive (\ref{main-lem-1}) by noting the bound in Lemma \ref{lem-2.1} and the fact that  $a_{2, \fp}$ is an integer.   
\end{proof}

Now we come to the second lemma, which provides a necessary condition for degree 1 primes $\fp\in \Sigma_K$  such that the reduction $\overline{A}_{\fp}/\F_p$ splits when $F :=\End_{\overline{K}}(A)\otimes_{\Z} \Q$ is isomorphic to a real quadratic field.

\begin{lemma}\label{main-lem-RM}
Let $A/K$ be an abelian surface such that $\End_{\overline{K}}(A)$ isomorphic to an order in  $F$. For a degree 1 prime $\fp\in \Sigma_K$ such that $\fp\nmid N_A$, if $\overline{A}_{\fp}$ splits, then one of the following holds:
\begin{enumerate}
\item $P_{A, \fp}(X) =(X^2+b_{1}X+p)^2 \in \Z[X], |b_1|\leq 2\sqrt{p}$, 
\item $P_{A, \fp}(X) =(X^2+b_{1}X+p)(X^2-b_{1}X+p) \in \Z[X], |b_1|\leq 2\sqrt{p}$.
\end{enumerate}
 If moreover $ \End_{\overline{K}}(A) =  \End_{K}(A)$ and $\overline{A}_{\fp}$ splits, then $\overline{A}_{\fp}$   splits into the square of an elliptic curve.  In particular,  
 $P_{A, \fp}(X) =(X^2+b_{1}X+p)^2 \in \Z[X], |b_1|\leq 2\sqrt{p}$.
\end{lemma}
\begin{proof}
By \cite[Proposition 4.3, p. 258]{Si92}, there exists a field extension $K'/K$ 
with $[K':K]\leq 2$ such that $ \End_{\overline{K'}}(A) =  \End_{K'}(A)$. For a prime $\fp$ in $K$ of degree 1 such that  $\fp\nmid N_A$ for which $\overline{A}_{\fp}$ splits, we denote by $\fp'$ a prime in $K'$ lying above $\fp$. Then, we have $[\F_{\fp'}: \F_{p}]=$ 1 or 2. Recall that there is an injection of endomorphism rings:
\[
\End_{K'}(A) \hookrightarrow  \End_{\F_{\fp'}}(\overline{A}_{\fp'}).
\]
In particular, $\End_{\F_{\fp'}}(\overline{A}_{\fp'})\otimes_{\Z} \Q$ contains the field $F$. Since $\overline{A}_{\fp}$ splits, then there exist two elliptic curves $E_1$, $E_2$
defined over $\F_{p}$ such that $\overline{A}_{\fp}$ is isogenous over $\F_{p}$ to $E_1\times E_2$. Suppose $E_1$ and $E_2$ are not $\F_{\fp'}$-isogenous, then there is an isomorphism of endomorphism rings
\[
\End_{\F_{\fp'}}(\overline{A}_{\fp'})\simeq \End_{\F_{\fp'}}(E_1)\times  \End_{\F_{\fp'}}(E_2).
\]
Recall that $\End_{\F_{\fp'}}(E_1)$ and $\End_{\F_{\fp'}}(E_2)$ are isomorphic to orders in an imaginary quadratic field or orders in a definite quaternion algebra over $\Q$. This contradicts with $\End_{\F_{\fp'}}(\overline{A}_{\fp'})$ containing an order in $F$. Therefore,  $\overline{A}_{\fp}$ is $\F_{p}$-isogenous to $E_1\times E_2$ where $E_1, E_2$ are $\F_{\fp'}$-isogenous.

First, we assume $[\F_{\fp'}: \F_{p}]=2$. Since $\overline{A}_{\fp}$ is isogenous over $\F_{\fp'}$ to the square of an elliptic curve over $\F_{p}$,  the polynomial $P_{A, \fp'}(X)$ has repeated roots. In particular, $\Delta_{A, \fp'}=0$. By \cite[Lemma 2.13, p. 325]{MaNa02}, we have 
\[
a_{1, \fp}^2\Delta_{A, \fp}=\Delta_{A, \fp'}=0.
\]
 If $a_{1, \fp}\neq 0$, then we get $\Delta_{A, \fp}=0$, so $P_{A, \fp}(X)=(X^2+b_{1}X+p)^2 \in \Z[X], |b_1|\leq 2\sqrt{p}$. If $a_{1, \fp}=0$, then 
$P_{A, \fp}(X)=X^4+a_{2, \fp}X^2+p^2$. As $\overline{A}_{\fp}$ splits,  $P_{A, \fp}(X)$ is the product of two quadratic polynomials. Therefore,  we deduce that   $P_{A, \fp}(X) =(X^2+b_{1}X+p)(X^2-b_{1}X+p) \in \Z[X], |b_1|\leq 2\sqrt{p}$.

Next, we assume $[\F_{\fp'}: \F_{p}]=1$. Then  $\overline{A}_{\fp}$ is isogenous to the square of an elliptic curve over $\F_{p}$. Hence 
 $P_{A, \fp}(X)  =(X^2+b_{1}X+p)^2 \in \Z[X], |b_1|\leq 2\sqrt{p}$. If  $ \End_{\overline{K}}(A) =  \End_{K}(A)$, then $\F_{\fp'}= \F_{p}$, so the last assertion follows immediately.

\end{proof}

Now we come to the last case where $\End_{\overline{K}}(A)\otimes_{\Z} \Q$ is isomorphic to a CM quartic field $F$. In particular,  $F$ contains a real quadratic field. If we also assume 
$ \End_{\overline{K}}(A) = \End_{K}(A)$, then Lemma \ref{main-lem-RM}  applies to obtain the following result. One can also find a proof in  \cite[Proposition 2.1, p. 7]{MuPa08}.

\begin{lemma}\label{main-lem-CM}
Let $A/K$ be an abelian surface such that $ \End_{\overline{K}}(A) = \End_{K}(A)$ isomorphic to an order in  $F$. For a prime $\fp$ in $K$ of degree 1 and $\fp\nmid N_A$, if $\overline{A}_{\fp}$ splits, then 
 $P_{A, \fp}(X) =(X^2+b_{1}X+p)^2 \in \Z[X], |b_1|\leq 2\sqrt{p}$.
\end{lemma}

\bigskip

\section{Proof of the Theorems and Corollaries}
\subsection{The square sieve}
In this section, we introduce the square sieve method which will be applied to the problem of split reductions of an abelian surface with a trivial geometric endomorphism ring.

We first recall the square sieve for multisets.
\begin{theorem}\label{sq-seive}
Let $\cal{A}$ be a finite set of  integers (not necessarily distinct)  and let $\cal{P}$ be a set of distinct rational primes. Assume 
$\displaystyle \max_{\substack{0\neq \alpha\in \cal{A}}}\{\log |\alpha|\}< |\cal{P}|$. Set
$
\cal{S}(\cal{A})\coloneqq\{\alpha\in \cal{A}: \alpha \text{ is a nonzero square}\}.
$
Then
\[
|\cal{S}(\cal{A})|\ll \frac{|\cal{A}|}{|\cal{P}|}+\max_{\substack{\ell_1\neq \ell_2\\\ell_1, \ell_2\in \cal{P}}}\left|\sum_{\alpha\in \cal{A}}\left(\frac{\alpha}{\ell_1\ell_2}\right)\right|.
\]
\end{theorem}
\begin{proof}
This based on the upper and lower bound estimations of $\displaystyle \sum_{\alpha\in \mathcal{A}}\left(\sum_{\ell\in \mathcal{P}}\left(\frac{\alpha}{\ell}\right) \right)^2$. See \cite[Theorem 1, p. 251]{He84} or \cite[Theorem 18, p. 1557]{CoDa08} for details. 
\end{proof}

We keep the notation  in the previous sections. Namely, let $A/K$  be an abelian surface with $\End_{\overline{K}}(A)\simeq \Z$.  By  Serre's Open Image Theorem for abelian surfaces \cite[Th\'{e}or\`{e}me 3]{Se86}, there is a  constant $C_K(A)$,  that only depends on $A$ and $K$, such that for each integer $m$ with $(m, C_K(A))=1$, we have 
$
\bar{\rho}_{A, m}(\Gal(\overline{K}/K))= G(m).
$ 
For any prime $\fp\in \Sigma_K$ such that $\fp \nmid N_A$, recall the characteristic polynomial $P_{A, \fp}(X)$ of $\overline{A}_{\fp}$ as defined in (\ref{char-poly-surface}). The \emph{discriminant} of this polynomial is defined to be  $\Delta_{A, \fp} = a_{1, \fp}^2 - 4a_{2, \fp} + 8q$, where $q = N(\fp)$.

Let $x > 2$ be a fixed real number. In the application of Theorem \ref{sq-seive}, we consider the following sets:
\begin{itemize}
\item We denote by $\cal{A}\coloneqq\{\Delta_{A, \fp}: \fp \in \Sigma_K, N(\fp)=p \leq x, \fp\nmid N_A\}$. Here we view $\cal{A}$ as a multiset, namely, if there are $\fp_1\neq \fp_2$ such that $\Delta_{A, \fp_1}=\Delta_{A, \fp_2}$, we view $\Delta_{A, \fp_1}, \Delta_{A, \fp_2}$ as distinct elements in $\cal{A}$. Inside $\cal{A}$, we only consider the discriminant $\Delta_{A, \fp}$ where the $\fp$ is a degree 1 prime.
Then, the Prime Number Theorem implies that as $x\to \infty$, 
$
|\cal{A}|\ll \frac{x}{\log x}.
$ Moreover, by Lemma \ref{lem-2.1}, we have the bound $|\Delta_{A, \fp}| \leq 32p$.
\item Let $z=z(x)>0$ be a real number depending on $x$ such that $z\geq (\log x)^{100}$. We denote by  $\cal{P}\coloneqq\{\ell \text{  prime}: \max\{C_K(A), z\}< \ell< 2z,  \ell\nmid N_A \}$. We will view the set as a set of parameters and choose $z=z(x)$ later
to optimize our bounds. Notice that for each $m\in \Z$ consisting of prime divisors in $\cal{P}$, we have $\Im \bar{\rho}_{A, m}=G(m)$. Again, the Prime Number Theorem implies that, for sufficiently large $z$, 
$
|\cal{P}|\asymp \frac{z}{\log z}.
$ 
\item We denote by $\cal{S}(\cal{A})\coloneqq\{\Delta_{A, \fp}\in \cal{A}: \Delta_{A, \fp}\text{ is a nonzero square in } \Z\}$,  viewed also as a multiset. Then 
$
|\cal{S}(\cal{A})|=\#\{\fp \in \Sigma_K: N(\fp)=p\leq x, \fp\nmid N_A, \Delta_{A, \fp} \text{ is a nonzero square in } \Z \}.
$
\end{itemize}

Based on the above discussion, we have the following estimation of   $|\cal{S}(\cal{A})|$, which is going to be the main term in the estimation of $\pi_{A, \text{split}}(x)$ when $\End_{\overline{K}}(A)\simeq \Z$.
\begin{proposition}\label{estimate-1}
We keep the same notation as above, 
then
\[\ds |\cal{S}(\cal{A})| \ll \frac{x\log z}{z\log x} + \max_{\substack{\ell_1\neq \ell_2\\ \ell_1, \ell_2\in \cal{P}}}\left|\sum_{\substack{N(\fp)=p\leq x\\\fp\nmid N_A}}\left(\frac{\Delta_{A, \fp}}{\ell_1\ell_2} \right)\right|.\]
\end{proposition}

\subsection{Proof of Theorem \ref{main-thm}}
We keep  the notation from earlier sections and assume $A/K$ is an absolutely simple abelian surface with $\End(A_{\overline{K}})\simeq \Z$. 

By Lemma \ref{main-lem}, we have
\begin{align}
\pi_{A, \text{split}}(x) &\leq \pi'_{A, \text{split}}(x)   \nonumber \\
& = \#\{\fp \in \Sigma_K: N(\fp) \leq x, \fp\nmid N_A, \overline{A}_{\fp} \text{ is not absolutely simple}\} \nonumber \\
& = \#\{\fp \in \Sigma_K: N(\fp) \leq x, \fp \text{ has degree 1, } \fp\nmid N_A, \overline{A}_{\fp} \text{ is not absolutely simple}\} \nonumber \\
& \hspace{0.5cm}  + \#\{\fp \in \Sigma_K: N(\fp) \leq x, \fp \text{ has degree greater than 1, } \fp\nmid N_A, \overline{A}_{\fp} \text{ is not absolutely simple}\} \nonumber \\
& \leq \#\{ \fp \in \Sigma_K, N(\fp)=p \leq x, \fp\nmid N_A, \Delta_{A, \fp}\neq 0 \text{ is a square in } \Z\} \label{sq-estimate-1}\\
&\hspace{0.5cm} + \#\{\fp \in \Sigma_K: N(\fp)=p \leq x, \fp\nmid N_A, \Delta_{A, \fp}=0\} \label{estimate-0}\\
& \hspace{0.5cm} + \#\{\fp \in \Sigma_K: N(\fp)=p \leq x, \fp\nmid N_A, a_{2, \fp}=ip, -2\leq i\leq  6, i\in \Z\} \label{estimate-2}\\
& \hspace{0.5cm} + \#\{\fp \in \Sigma_K: N(\fp)=p \leq x, \fp\nmid N_A, a_{1, \fp}=0\} \label{estimate-3}\\
&\hspace{0.5cm} +  \#\{\fp \in \Sigma_K: N(\fp)=p\leq x, \fp\nmid N_A, a_{1, \fp}^2=p+a_{2, \fp}\} \label{estimate-4}\\
&\hspace{0.5cm} +  \#\{\fp \in \Sigma_K: N(\fp)=p \leq x, \fp\nmid N_A, a_{1, \fp}^2=2a_{2, \fp}\} \label{estimate-5}\\
&\hspace{0.5cm} +  \#\{\fp \in \Sigma_K: N(\fp)=p \leq x: \fp\nmid N_A, a_{1, \fp}^2=3(a_{2, \fp}-p)\}\label{estimate-6}\\
& \hspace{0.5cm} + \#\{\fp \in \Sigma_K: N(\fp) \leq x: \fp \text{ has degree greater than 1}\} \label{estimate-7}
\end{align}

We will show that the equation  (\ref{sq-estimate-1}) is the main term. To  bound  (\ref{sq-estimate-1}), we apply the square sieve together with various effective Chebotarev Density Theorem for properly chosen subfield extensions. To bound (\ref{estimate-0})-(\ref{estimate-6}), we  use variations of the  effective Chebotarev Density Theorem for division fields of $A/K$.  Note that the equation (\ref{estimate-7}) is  $\O(x^{\frac{1}{2}})$ (see e.g., \cite[p. 138]{Se81}).

We observe that the prime counting function in (\ref{sq-estimate-1}) is equal to $|\cal{S}(\cal{A})|$. In view of Proposition \ref{estimate-1}, for any fixed primes  $\ell_1, \ell_2\in \cal{P}$ such that $\ell_1\neq \ell_2$,  we aim to  bound
$
\ds\left|\sum_{\substack{N(\fp)=p\leq x\\ \fp\nmid N_A}}\left( \frac{\Delta_{A, \fp}}{\ell_1\ell_2}\right)\right|.
$
We write
\begin{align}\label{S-pi-bound}
S(\ell_1\ell_2) & := \sum_{\substack{N(\fp)=p \leq x\\\fp\nmid N_A}}\left( \frac{\Delta_{A, \fp}}{\ell_1\ell_2}\right) =\sum_{\substack{N(\fp)=p\leq x, \fp\nmid N_A\\
\left(\frac{\Delta_{A, \fp}}{\ell_1} \right)\left(\frac{\Delta_{A, \fp}}{\ell_2} \right)=1}} 1-\sum_{\substack{N(\fp)=p \leq x, \fp\nmid N_A\\ \left(\frac{\Delta_{A, \fp}}{\ell_1} \right)\left(\frac{\Delta_{A, \fp}}{\ell_2} \right)=-1}} 1,\\  \nonumber
& =\pi_{A, 1}(x)-\pi_{A, -1}(x),
\end{align}
where
\begin{align}
\pi_{A, 1}(x) & \coloneqq\#\{\fp \in \Sigma_K: N(\fp)=p \leq x, \fp\nmid N_A, \left(\frac{\Delta_{A, \fp}}{\ell_1} \right)\left(\frac{\Delta_{A, \fp}}{\ell_2} \right)=1\}, \label{1-pi-counting}\\
\pi_{A, -1}(x) & \coloneqq\#\{\fp \in \Sigma_K: N(\fp)=p \leq x, \fp\nmid N_A, \left(\frac{\Delta_{A, \fp}}{\ell_1} \right)\left(\frac{\Delta_{A, \fp}}{\ell_2} \right)=-1\}. \label{-1-pi-counting}
\end{align}

Now we move on to find appropriate Chebotarev counting functions to bound  (\ref{1-pi-counting}) and (\ref{-1-pi-counting}).
Fix a prime $\fp\in \Sigma_K$ of degree 1 such that $\fp\nmid \ell_1\ell_2 N_A$. Recall that by  Serre's Open Image Theorem, for any $\ell_1, \ell_2\in \cal{P}$, we have 
$
\Im \overline{\rho}_{A, \ell_1\ell_2} \simeq \Gal(K(A[\ell_1\ell_2])/K) 
\simeq G(\ell_1\ell_2).
$ 
 Moreover,  the characteristic polynomial of the matrix  $\bar{\rho}_{A, \ell_1\ell_2}\left(\left(\frac{K(A[\ell_1\ell_2])/K}{\fp} \right)\right)$ is 
 \[
P_{A, \fp}(X) \equiv X^4+a_{1, \fp}X^3+a_{2, \fp}X^2+pa_{1, \fp}X+p^2 (\mod \ell_1\ell_2).
\]
In particular,  the following relation between the discriminants 
holds: 
\begin{equation}\label{disc-relation}
\Delta_{\bar{\rho}_{A, \ell_1\ell_2}\left(\left(\frac{K(A[\ell_1\ell_2])/K}{\fp} \right)\right)}\equiv \Delta_{A, \fp} (\mod \ell_1\ell_2).
\end{equation} 
To  find appropriate Chebotarev counting functions to bound   $\pi_{A, 1}(x)$ and $\pi_{A, -1}(x)$, it is natural to consider the Galois extension $K(A[\ell_1\ell_2])/K$ 
 with Galois group $G(\ell_1\ell_2)$
and  take
the following conjugation invariant sets:
\begin{align*}
\cal{C}'_+(\ell_1\ell_2) & \coloneqq\{g\in G(\ell_1\ell_2): \left(\frac{\Delta_{g}}{\ell_1} \right) \left(\frac{\Delta_{g}}{\ell_2} \right)=1\},\\
\cal{C}'_-(\ell_1\ell_2)&  \coloneqq\{g\in G(\ell_1\ell_2): \left(\frac{\Delta_{g}}{\ell_1} \right) \left(\frac{\Delta_{g}}{\ell_2} \right)=-1\}.
\end{align*}
Indeed, we have 
\begin{align*}
\pi_{A, 1}(x)=\pi_{\cal{C}'_+(\ell_1\ell_2)}(x, K(A[\ell_1\ell_2]/K))+\O(x^{\frac{1}{2}}),\\
\pi_{A, -1}(x)=\pi_{\cal{C}'_-(\ell_1\ell_2)}(x, K(A[\ell_1\ell_2]/K))+\O(x^{\frac{1}{2}}).
\end{align*}

However, motivated by the strategy of   Murty-Murty-Saradha \cite{MuMuSa88} and Cojocaru-David \cite{CoDa08}, we search for Chebotarev counting functions constructed from 
a subfield extensions of $K(A[\ell_1\ell_2])/K$  for better bounds. We find that working  in the field extension $K(A[\ell_1\ell_2])^{\Lambda(\ell_1\ell_2)}/K$ with 
$
 \Gal(K(A[\ell_1\ell_2])^{\Lambda(\ell_1\ell_2)}/K)\simeq P(\ell_1\ell_2)
$
does the work.
Fixing a lifting $\hat{g}\in G(\ell_1\ell_2)$ of $g$ for each $g\in P(\ell_1\ell_2)$ and recalling the relation (\ref{quadratic-P-G}), we define the modified conjugacy classes in $P(\ell_1\ell_2)$ as follows:
\begin{align*}
\cal{C}_+(\ell_1\ell_2) & \coloneqq\{g\in P(\ell_1\ell_2): \left(\frac{\Delta_{\hat{g}}}{\ell_1} \right) \left(\frac{\Delta_{\hat{g}}}{\ell_2} \right)=1\},\\
\cal{C}_-(\ell_1\ell_2)&  \coloneqq\{g\in P(\ell_1\ell_2): \left(\frac{\Delta_{\hat{g}}}{\ell_1} \right) \left(\frac{\Delta_{\hat{g}}}{\ell_2} \right)=-1\}.
\end{align*}
Clearly, $\cal{C}_{\pm}(\ell_1\ell_2)$ are the images of  $\cal{C}'_{\pm}(\ell_1\ell_2)$ under the projection $G(\ell_1\ell_2)/\Lambda(\ell_1\ell_2)\twoheadrightarrow P(\ell_1\ell_2)$. Note  that by (\ref{quadratic-P-G}), we have 
\[
\Lambda(\ell_1\ell_2) \cal{C}'_+(\ell_1\ell_2)\subseteq \cal{C}'_+(\ell_1\ell_2), \; \Lambda(\ell_1\ell_2) \cal{C}'_-(\ell_1\ell_2)\subseteq \cal{C}'_-(\ell_1\ell_2).
\] 
Therefore, a functorial comparison result  \cite[Lemma 2.6 (ii)]{Zy15}  between the Chebotarve counting functions gives:
\[
\pi_{\cal{C}'_+(\ell_1\ell_2)}(x, K(A[\ell_1\ell_2])/K) =\pi_{\cal{C}_+(\ell_1\ell_2)}(x, K(A[\ell_1\ell_2])^{\Lambda(\ell_1\ell_2)}/K), 
\]
\[
 \pi_{\cal{C}'_-(\ell_1\ell_2)}(x, K(A[\ell_1\ell_2])/K) = \pi_{\cal{C}_-(\ell_1\ell_2)}(x, K(A[\ell_1\ell_2])^{\Lambda(\ell_1\ell_2)}/K)).
\]
Observe that if a prime  $\fp\nmid \ell_1\ell_2 N_A$ is  counted by $\pi_{A, 1}(x) $ and 
$\pi_{A, -1}(x)$,
then by (\ref{disc-relation}), we have 
\[
\bar{\rho}_{A, \ell_1\ell_2}\left(\left(\frac{K(A[\ell_1\ell_2])/K}{\fp} \right)\right)\subseteq \cal{C}'_+(\ell_1\ell_2),
\]
\[
 \bar{\rho}_{A, \ell_1\ell_2}\left(\left(\frac{K(A[\ell_1\ell_2])/K}{\fp} \right)\right)\subseteq \cal{C}'_{-}(\ell_1\ell_2).
 \]
 The converse is also true for degree 1 primes.
Therefore, 
\begin{align}\label{transfer-to-Cheb}
S(\ell_1\ell_2) & =\pi_{A, 1}(x)-\pi_{A, -1}(x) \nonumber \\
& =\pi_{\cal{C}_+(\ell_1\ell_2)}(x, K(A[\ell_1\ell_2])^{\Lambda(\ell_1\ell_2)}/K) - \pi_{\cal{C}_-(\ell_1\ell_2) }(x, K(A[\ell_1\ell_2])^{\Lambda(\ell_1\ell_2)}/K) +\O_K(x^{\frac{1}{2}}). 
\end{align}

Before presenting the proof, we want to point out that it is the following two key steps that help advance the upper bounds. One is splitting the character sum $S(\ell_1\ell_2)$ into $\pi_{A, 1}(x)$ and $\pi_{A, -1}(x)$ as shown in (\ref{S-pi-bound}), which gives us  a cancellation of the main terms between the two functions based on the following computations 
 (see Proposition \ref{countings} (ii)): 
\begin{align}
|\cal{C}_+(\ell_1\ell_2)| &= |\cal{C}_1(\ell_1)||\cal{C}_1(\ell_2)|+  |\cal{C}_{-1}(\ell_1)||\cal{C}_{-1}(\ell_2)|  \nonumber \\
&= \frac{1}{2}(\ell_1\ell_2)^{10}-\frac{1}{2}\ell_1^9\ell_2^{10} -\frac{1}{2}\ell_1^{10}\ell_2^9+\O((\ell_1\ell_2)^9+\ell_1^8\ell_2^{10}+\ell_1^{10}\ell_2^8), \label{1-size-conj}\\
|\cal{C}_-(\ell_1\ell_2) | &=  |\cal{C}_1(\ell_1)||\cal{C}_{-1}(\ell_2)|+  |\cal{C}_{-1}(\ell_1)||\cal{C}_{1}(\ell_2)| \nonumber \\
& = \frac{1}{2}(\ell_1\ell_2)^{10}-\frac{1}{2}\ell_1^9\ell_2^{10} -\frac{1}{2}\ell_1^{10}\ell_2^9+\O((\ell_1\ell_2)^9+\ell_1^8\ell_2^{10}+\ell_1^{10}\ell_2^8). \label{-1-size-conj}
\end{align}
The other is considering the subfield extension $K(A[\ell_1\ell_2])^{\Lambda(\ell_1\ell_2)}/K$ (instead of $K(A[\ell_1\ell_2])/K$), which 
 leads to a smaller size of the conjugation invariant sets in $\Gal(K(A[\ell_1\ell_2])^{\Lambda(\ell_1\ell_2)}/K)$ and hence a smaller error term while applying  the effective  versions of the Chebotarev Density Theorem.

Now we move on to the proof of  Theorem \ref{main-thm}.
\begin{proof}
\begin{enumerate}
\item
Let $\ell_1, \ell_2\in \mathcal{P}$ as before.
Invoking Proposition \ref{estimate-1}, (\ref{transfer-to-Cheb}), Theorem \ref{effective-CDT} (i) with 
$K(A[\ell_1\ell_2])^{\Lambda(\ell_1\ell_2)}/K$, $G=P(\ell_1\ell_2)$, and $\cal{C}=\cal{C}_{\pm}(\ell_1\ell_2)$, Proposition \ref{countings} (i) with $m=\ell_1\ell_2$, Lemma \ref{hensel},  the criterion of N\'{e}ron-Ogg-Shafarevich for abelian varieties, and  (\ref{1-size-conj})-(\ref{-1-size-conj}), we get
\begin{align*}
 |\cal{S}(\cal{A})|& \ll_K \frac{x\log z}{z\log x}+x^{\frac{1}{2}}+ \max_{\substack{\ell_1\neq \ell_2\\ \ell_1, \ell_2\in \cal{P}}}\left| \pi_{\cal{C}_+(\ell_1\ell_2)}(x, K(A[\ell_1\ell_2])^{\Lambda(\ell_1\ell_2)}/K)- \pi_{\cal{C}_-(\ell_1\ell_2)}(x,  K(A[\ell_1\ell_2])^{\Lambda(\ell_1\ell_2)}/K)  \right|\\
& \ll_K \frac{x\log z}{z\log x}+x^{\frac{1}{2}} + \max_{\substack{\ell_1\neq \ell_2\\ \ell_1, \ell_2\in \cal{P}}} \left| \left(\frac{|\cal{C}_+(\ell_1\ell_2)|}{|P(\ell_1\ell_2)|}-\frac{|\cal{C}_-(\ell_1\ell_2)|}{|P(\ell_1\ell_2)|}\right)\Li(x) \right.\\
&\hspace{4cm} \left. + \left(|\cal{C}_+(\ell_1\ell_2)|+|\cal{C}_-(\ell_1\ell_2)|\right)x^{\frac{1}{2}}\left(\frac{\log |d_{K(A[\ell_1\ell_2])^{\Lambda(\ell_1\ell_2)}}|}{|P(\ell_1\ell_2)|}+\log(x)\right) \right|\\
&\ll_K \frac{x\log z}{z\log x}+ x^{\frac{1}{2}}+ \max_{\substack{\ell_1\neq \ell_2\\ \ell_1, \ell_2\in \cal{P}}} \left|\left(\frac{1}{\ell_1}+\frac{1}{\ell_2}\right)^2\frac{x}{\log x} +(\ell_1\ell_2)^{10}x^{\frac{1}{2}}\log(\ell_1\ell_2 N_A x)\right|.
\end{align*}

Note that since  $z<\ell_1, \ell_2<2z$, 
\[
|\cal{S}(\cal{A})| \ll_{A, K}  \frac{x\log z}{z\log x} +z^{20}x^{\frac{1}{2}}\log x.
\]
By taking $z=z(x)\sim \frac{x^{\frac{1}{42}}}{(\log x)^{\frac{1}{21}}}$, we obtain that for all sufficiently large $x$, 
\[
|\cal{S}(\cal{A})|\ll_{A, K} x^{\frac{41}{42}}(\log x)^{\frac{1}{21}}.
\]
Let $\ell> C_K(A)$ be a prime. The estimations for (\ref{estimate-0})-(\ref{estimate-6}) are similar, and are direct applications of Proposition \ref{countings} (i) and (iv) and Theorem \ref{effective-CDT} (i) with  $K(A[\ell])/K$, $G=G(\ell)$, and $\cal{C}$ to be one of (\ref{conjugacy-0})-(\ref{conjugacy-5}). Then  as $x\to \infty$,  each of the terms in (\ref{estimate-0})-(\ref{estimate-6}) is 
\[
\O_{A, K} \left(\frac{x}{\ell\log x}+\ell^{10}x^{\frac{1}{2}}\log(\ell N_A x)\right).
 \]
 By taking $\ell=\ell(x)\asymp \frac{x^{\frac{1}{22}}}{(\log x)^{\frac{2}{11}}}$, we obtain that each of the terms in (\ref{estimate-0})-(\ref{estimate-6}) is   $\O_{A, K}\left( \frac{x^{\frac{21}{22}}}{(\log x)^{\frac{9}{11}}}\right)$.
 In all, we obtain that for all sufficiently large $x$, 
 \[
 \pi_{A, \text{split}}(x) \ll_{A, K} x^{\frac{41}{42}}(\log x)^{\frac{1}{21}}.
 \]
 
 \item
Let $\ell_1, \ell_2\in \mathcal{P}$ as before. Invoking Proposition \ref{estimate-1} (\ref{transfer-to-Cheb}), Theorem \ref{effective-CDT} (ii) with $K(A[\ell_1\ell_2])^{\Lambda(\ell_1\ell_2)}/K$, $G=P(\ell_1\ell_2)$, and $\cal{C}=\cal{C}_{\pm}(\ell_1\ell_2)$, Proposition \ref{countings} (i) with $m=\ell_1\ell_2$, Lemma \ref{hensel},  the criterion of N\'{e}ron-Ogg-Shafarevich for abelian varieties, and (\ref{1-size-conj})-(\ref{-1-size-conj}), we obtain
\begin{align*}
 |\cal{S}(\cal{A})| & \ll_K \frac{x\log z}{z\log x}+x^{\frac{1}{2}} +  \max_{\substack{\ell_1\neq \ell_2\\ \ell_1, \ell_2\in \cal{P}}}\left| \pi_{\cal{C}_+(\ell_1\ell_2)}(x, K(A[\ell_1\ell_2])^{\Lambda(\ell_1\ell_2)}/K)- \pi_{\cal{C}_-(\ell_1\ell_2)}(x,  K(A[\ell_1\ell_2])^{\Lambda(\ell_1\ell_2)}/K)  \right|\\
& \ll_K \frac{x\log z}{z\log x}+ x^{\frac{1}{2}}+  \max_{\substack{\ell_1\neq \ell_2\\ \ell_1, \ell_2\in \cal{P}}}\left|\left(\frac{|\cal{C}_+(\ell_1\ell_2)|}{|P(\ell_1\ell_2)|}-\frac{|\cal{C}_-(\ell_1\ell_2)|}{|P(\ell_1\ell_2)|} \right)\Li(x) \right.\\
&\hspace{4cm}\left. +\left(|\cal{C}_+(\ell_1\ell_2)|^{\frac{1}{2}}+|\cal{C}_-(\ell_1\ell_2)|^{\frac{1}{2}}\right) x^{\frac{1}{2}} \log\left(M(K(A[\ell_1\ell_2])^{\Lambda(\ell_1\ell_2)}/K) x\right) \right|\\
&\ll_K \frac{x\log z}{z\log x}+ x^{\frac{1}{2}}+ \max_{\substack{\ell_1\neq \ell_2\\ \ell_1, \ell_2\in \cal{P}}}\left|\left(\frac{1}{\ell_1}+\frac{1}{\ell_2}\right)^2\frac{x}{\log x} +(\ell_1\ell_2)^{5}x^{\frac{1}{2}}\log\left((\ell_1\ell_2)N_Ax\right)\right|.
\end{align*}

Note that since  $z<\ell_1, \ell_2<2z$, 
\[
|\cal{S}(\cal{A})| \ll_{A, K}  \frac{x\log z}{z\log x} +z^{10}x^{\frac{1}{2}}\log x.
\]
By taking $z=z(x)\sim \frac{x^{\frac{1}{22}}}{(\log x)^{\frac{1}{11}}}$, we obtain that for all sufficiently large $x$, 
\[
|\cal{S}(\cal{A})|\ll_{A, K} x^{\frac{21}{22}}(\log x)^{\frac{1}{11}}.
\]
Let $\ell> C_K(A)$ be a prime. Similarly as before, by applying Proposition \ref{countings} (i) and 
(iv) and Theorem \ref{effective-CDT} (ii) with $L=K(A[\ell])$, $G=G(\ell)$, and $\cal{C}$ to be one of (\ref{conjugacy-0})-(\ref{conjugacy-5}),  we obtain that (\ref{estimate-0})-(\ref{estimate-6}) are
 \[
\O_{A, K}\left(\frac{x}{\ell\log x}+\ell^5x^{\frac{1}{2}}\log(\ell N_A x)\right).
 \]
Taking $\ell=\ell(x)\asymp \frac{x^{\frac{1}{12}}}{(\log x)^{\frac{1}{3}}}$, we obtain that each of the terms in  (\ref{estimate-0})-(\ref{estimate-6}) is  $\O_{A, K}\left(\frac{x^{\frac{11}{12}}}{(\log x)^{\frac{2}{3}}}\right)$.
 Thus, we obtain that for all sufficiently large $x$, 
 \[
 \pi_{A, \text{split}}(x) \ll_{A, K} x^{\frac{21}{22}}(\log x)^{\frac{1}{11}}.
 \]

\item
Let $\ell_1, \ell_2\in \mathcal{P}$ as before.
Invoking Proposition \ref{estimate-1}, (\ref{transfer-to-Cheb}), Theorem \ref{effective-CDT} (iii) with $K(A[\ell_1\ell_2])^{\Lambda(\ell_1\ell_2)}/K$, $G=P(\ell_1\ell_2)$, and $\cal{C}=\cal{C}_{\pm}(\ell_1\ell_2)$, Proposition \ref{countings} (i) and (iii) with $m=\ell_1\ell_2$, Lemma \ref{hensel},  the criterion of N\'{e}ron-Ogg-Shafarevich for abelian varieties, and  (\ref{1-size-conj})-(\ref{-1-size-conj}), we obtain
\begin{align*}
 |\cal{S}(A)| & \ll_K \frac{x\log z}{z\log x}+x^{\frac{1}{2}} +  \max_{\substack{\ell_1\neq \ell_2\\ \ell_1, \ell_2\in \cal{P}}}\left|\pi_{\cal{C}_+(\ell_1\ell_2)}(x, K(A[\ell_1\ell_2])^{\Lambda(\ell_1\ell_2)}/K)- \pi_{\cal{C}_-(\ell_1\ell_2)}(x,  K(A[\ell_1\ell_2])^{\Lambda(\ell_1\ell_2)}/K)  \right|\\
& \ll_K \frac{x\log z}{z\log x}+ x^{\frac{1}{2}}+ \max_{\substack{\ell_1\neq \ell_2\\ \ell_1, \ell_2\in \cal{P}}}\left|\left(\frac{|\cal{C}_+(\ell_1\ell_2)|}{|P(\ell_1\ell_2)|}-\frac{|\cal{C}_-(\ell_1\ell_2)|}{|P(\ell_1\ell_2)|} \right)\Li(x) \right.\\
&\hspace{4cm}\left. + \left(|\cal{C}_+(\ell_1\ell_2)|^{\frac{1}{2}}+|\cal{C}_-(\ell_1\ell_2)|^{\frac{1}{2}}\right)\left(\frac{|P(\ell_1\ell_2)^{\#}|}{|P(\ell_1\ell_2)|} \right)^{\frac{1}{2}}x^{\frac{1}{2}}\log \left(M(K(A[\ell_1\ell_2])^{\Lambda(\ell_1\ell_2)}/K)x\right)  \right|\\
&\ll_K \frac{x\log z}{z\log x}+ x^{\frac{1}{2}}+\max_{\substack{\ell_1\neq \ell_2\\ \ell_1, \ell_2\in \cal{P}}}\left|\left(\frac{1}{\ell_1}+\frac{1}{\ell_2}\right)^2\frac{x}{\log x} +(\ell_1\ell_2)x^{\frac{1}{2}}\log\left((\ell_1\ell_2)N_Ax\right)\right|.
\end{align*}
Note that since  $z<\ell_1, \ell_2<2z$, 
\[
|\cal{S}(A)| \ll_{A, K} \frac{x\log z}{z\log x} +z^{2}x^{\frac{1}{2}}\log x.
\]
By taking $z=z(x)\sim \frac{x^{\frac{1}{6}}}{(\log x)^{\frac{1}{3}}}$, we obtain that for all sufficiently large $x$, 
\[
|\cal{S}(A)|\ll_{A, K} x^{\frac{5}{6}}(\log x)^{\frac{1}{3}}.
\]
Let $\ell> C_K(A)$ be a prime. By applying Proposition \ref{countings}  (i) and 
(iv) and Theorem \ref{effective-CDT} (i) with $K(A[\ell])/K$, $G=G(\ell)$, and $\cal{C}$ to be one of  (\ref{conjugacy-0})-(\ref{conjugacy-5}), we obtain that    (\ref{estimate-0})-(\ref{estimate-6}) are
 \[
\O_{A, K}\left(\frac{x}{\ell\log x}+\ell x^{\frac{1}{2}}\log(\ell N_A x)\right).
 \]
Taking $\ell=\ell(x)\asymp \frac{x^{\frac{1}{4}}}{\log x}$, we obtain that each of the terms in (\ref{estimate-0})-(\ref{estimate-6}) is  $\O_{A, K}(x^{\frac{3}{4}})$.
 Thus,  we obtain that for all sufficiently large $x$, 
 \[
 \pi_{A, \text{split}}(x) \ll_{A, K} x^{\frac{5}{6}}(\log x)^{\frac{1}{3}}.
 \]

\end{enumerate}
\end{proof}
\medskip

\subsection{Proof of Theorem \ref{main-thm-RM}}
Let $A/K$ be an absolutely simple abelian surface such that $\End_{\overline{K}}(A) \otimes_{\Z} \Q= \End_K(A) \otimes_{\Z} \Q $ is isomorphic to  a real quadratic field $F$.  
 Then, by Lemma \ref{main-lem-RM}, 
\begin{align}
\pi_{A, \text{split}}(x, t) & =\#\{\fp\in \Sigma_K: N(\fp) \leq x, \fp\nmid N_A, \overline{A}_{\fp} \text{ splits}\} \nonumber \\
& \ll \#\{\fp\in \Sigma_K: N(\fp)=p \leq x, \fp\nmid N_A, P_{A, \fp}(X)=(X^2+b_1X+p)^2, |b_1|\leq 2\sqrt{p}, b_1 \in \Z\}  \label{1-pi-counting-RM}\\
& \hspace{0.5cm}+ x^{\frac{1}{2}}. \nonumber 
\end{align}
By  \cite[Remark 1.6, p. 29]{Lo16}, there is an explicit constant $C_K'(A)$ such that for any prime $\ell>C_K'(A)$ that splits in $F$, we have
$
\Im \bar{\rho}_{A,\ell} = \G(\ell). 
$
We fix  a prime $\ell=\ell(x)$ for which $\ell>C_K'(A)$ and splits in $F$.

We now bound  (\ref{1-pi-counting-RM}) 
by a  Chebotarev counting function as follows. As in the previous subsection, we work over the subfield extension $K(A[\ell])^{\Lambda(\ell)}/K$ rather than $K(A[\ell])/K$. We denote by
\begin{align}
\cal{C}_1^{'RM}(\ell)& \coloneqq \{(g_1, g_2)\in \G(\ell): \tr g_1=\tr g_2\}. \label{conj-G-1-RM}
\end{align}
    Recall $\cal{C}_1^{RM}(\ell)$ 
    defined in  (\ref{conj-P-1-RM}) 
   is the image of $\cal{C}_1^{'RM}(\ell)$ 
    in $\P(\ell)$. 
 Note that by (\ref{conjugacy-2}), we have $\Lambda(\ell)\cal{C}_{ 1}^{'RM}(\ell)\subseteq \cal{C}_{1}^{'RM}(\ell)$. Again, we invoke \cite[Lemma 2.6 (ii)]{Zy15} and obtain 
    \[
    \pi_{\cal{C}_{ 1}^{'RM}(\ell)}(x, K(A[\ell])/K)=\pi_{\cal{C}_{ 1}^{RM}(\ell)}(x, K(A[\ell])^{\Lambda(\ell)}/K).
    \]
  Note that  if   $\fp\nmid \ell N_A$ is counted by (\ref{1-pi-counting-RM}) 
then
$\ds\bar{\rho}_{A, \ell}\left(\left(\frac{K(A[\ell])/K}{\fp} \right)\right)$ is contained in $\cal{C}_1^{'RM}(\ell)$ and the reverse also holds for degree 1 primes. 
Thus, we obtain
\begin{align}
\pi_{A, \text{split}}(x) & \ll_K   \pi_{\cal{C}_{1}^{'RM}(\ell)}(x, K(A[\ell])/K)+ \O(x^{\frac{1}{2}}) \nonumber \\ 
& =\pi_{\cal{C}_1^{RM}(\ell)}(x, K(A[\ell])^{\Lambda(\ell)}/K) +\O(x^{\frac{1}{2}}). \label{estimate-1-RM} 
\end{align}

Based on the observations above, we give the proof of Theorem \ref{main-thm-RM}.
\begin{proof}
\begin{enumerate}
\item Let $\ell$ be a prime such that $\ell> C_K'(A)$ and $\ell$  splits in $F$. Invoking (\ref{estimate-1-RM}), Theorem \ref{effective-CDT} (i) with $K(A[\ell])^{\Lambda(\ell)}/K$ and $G=\P(\ell)$, Proposition \ref{countings-2}  (i) and (ii), Lemma \ref{hensel}, and the criterion of N\'{e}ron-Ogg-Shafarevich for abelian varieties, we deduce that
\begin{align*}
\pi_{A, \text{split}}(x) &\ll_K x^{\frac{1}{2}}+ \frac{|\cal{C}^{RM}_1(\ell)|}{|\P(\ell)|}\Li(x) +|\cal{C}^{RM}_1(\ell)|x^{\frac{1}{2}}\left( \frac{\log |d_{K(A[\ell])^{\Lambda(\ell)}}|}{|\P(\ell)|}+\log x\right)\\
& \ll_K \frac{x}{\ell \log x}+\ell^5x^{\frac{1}{2}}\log(\ell N_A x).
\end{align*}
Taking $\ell=\ell(x)\asymp_{K, F} \frac{x^{\frac{1}{12}}}{(\log x)^{\frac{1}{3}}}$, we obtain that for all sufficiently large $x$, 
\[
\pi_{A, \text{split}}(x) \ll_{A, K, F} \frac{x^{\frac{11}{12}}}{(\log x)^{\frac{2}{3}}}.
\]
\item Let $\ell$ be a prime as before. Invoking (\ref{estimate-1-RM}), Theorem \ref{effective-CDT} (ii) with $K(A[\ell])^{\Lambda(\ell)}/K$ and $G=\P(\ell)$, Proposition \ref{countings-2}  (i) and (ii), Lemma \ref{hensel}, and the the criterion of N\'{e}ron-Ogg-Shafarevich for abelian varieties, we deduce that
\begin{align*}
\pi_{A, \text{split}}(x) &\ll_K x^{\frac{1}{2}}+ \frac{|\cal{C}^{RM}_1(\ell)|}{|\P(\ell)|}\Li(x) +|\cal{C}^{RM}_1(\ell)|^{\frac{1}{2}} x^{\frac{1}{2}}\left( \log (M(K(A[\ell])^{\Lambda(\ell)}/K) x)\right)\\
& \ll_K \frac{x}{\ell\log x}+\ell^{\frac{5}{2}}x^{\frac{1}{2}}\log(\ell N_A x).
\end{align*}
Taking $\ell=\ell(x)\asymp_{K, F} \frac{x^{\frac{1}{7}}}{(\log x)^{\frac{4}{7}}}$, we obtain that for all sufficiently large $x$, 
\[
\pi_{A, \text{split}}(x) \ll_{A, K, F} \frac{x^{\frac{6}{7}}}{(\log x)^{\frac{3}{7}}}.
\]
\item  Let $\ell$ be a prime as before. Invoking (\ref{estimate-1-RM}), Theorem \ref{effective-CDT} (iii) with $K(A[\ell])^{\Lambda(\ell)}/K$ and $G=\P(\ell)$, Proposition \ref{countings-2}  (i), (ii) and (iii), Lemma \ref{hensel}, and the the criterion of N\'{e}ron-Ogg-Shafarevich for abelian varieties, we deduce that
\begin{align*}
\pi_{A, \text{split}}(x) &\ll_K x^{\frac{1}{2}}+ \frac{|\cal{C}^{RM}_1(\ell)|}{|\P(\ell)|}\Li(x) +|\cal{C}^{RM}_1(\ell)|^{\frac{1}{2}} \left(\frac{|\P^{\#}(\ell)|}{|\P(\ell)|}\right)^{\frac{1}{2}}  x^{\frac{1}{2}} \log \left(M(K(A[\ell])^{\Lambda(\ell)}/K) x\right)\\
& \ll_K \frac{x}{\ell\log x}+\ell^{\frac{1}{2}} x^{\frac{1}{2}}\log(\ell N_A x).
\end{align*}
Taking $\ell=\ell(x)\asymp_{K, E} \frac{x^{\frac{1}{3}}}{(\log x)^{\frac{4}{3}}}$, we obtain that for all sufficiently large $x$, 
\[
\pi_{A, \text{split}}(x) \ll_{A, K, E} x^{\frac{2}{3}}\log x.
\]
\end{enumerate}
\end{proof}
\medskip

\subsection{Proof of Theorem \ref{main-thm-CM}}

Let $A/K$ be an absolutely simple abelian surface such that $ \End_{\overline{K}}(A)\otimes_{\Z} \Q= \End_{K}(A)\otimes_{\Z} \Q$  is isomorphic to a CM quartic field $F$. 
%
 Then, by Lemma \ref{main-lem-CM}, 
\begin{align}
\pi_{A, \text{split}}(x, t) & =\#\{\fp\in \Sigma_K: N(\fp) \leq x, \fp\nmid N_A, \overline{A}_{\fp} \text{ splits}\} \nonumber \\
& \ll_K \#\{\fp\in \Sigma_K: N(\fp)=p \leq x, \fp\nmid N_A, P_{A, \fp}(X)=(X^2+b_1X+p)^2, |b_1|\leq 2\sqrt{p}, b_1 \in \Z\}  \label{1-pi-counting-CM}\\
& \hspace{0.5cm}+ x^{\frac{1}{2}}. \nonumber 
\end{align}
By the assumptions on $\phi: T_{F'} \to T_{F}$ and $F\subseteq K$, there is a constant $C''_K(A)>0$ such that  for any prime $\ell=\ell(x)>C''_K(A)$ that splits in $F$, we have 
$
\Im \bar{\rho}_{A,\ell} = \T(\ell)
$
 \cite[Theorem A, p. 118 and Example 1, p. 120]{BaGaKr03}.
Fix a prime  $\ell=\ell(x)>C''_K(A)$ that splits in $F$. We bound the equation (\ref{1-pi-counting-CM}) by  a Chebotarev counting functions as follows. Consider the subfield extension $K(A[\ell])^{\Lambda(\ell)}/K$ and denote by
\begin{align}
\cal{C}^{'CM}(\ell)& \coloneqq \{(g_1, g_2)\in T(\ell): \tr g_1=\tr g_2\}.\label{conj-G-1-CM}
\end{align}
    Note the union conjugacy classes $\cal{C}^{CM}(\ell)$  defined in  (\ref{conj-P-1-CM}) 
   is the image of $\cal{C}^{'CM}(\ell)$  in $\T(\ell)$. 
 Observe that by (\ref{conjugacy-2}), we have $\Lambda(\ell)\cal{C}^{'CM}(\ell)\subseteq \cal{C}^{'CM}(\ell)$. Again, we invoke \cite[Lemma 2.6 (ii)]{Zy15} and obtain 
    \[
    \pi_{\cal{C}^{'CM}(\ell)}(x, K(A[\ell])/K)=\pi_{\cal{C}^{CM}(\ell)}(x, K(A[\ell])^{\Lambda(\ell)}/K).
    \]
  Note that  if   $\fp\nmid \ell N_A$ is counted by (\ref{1-pi-counting-CM}), 
then
$\ds\bar{\rho}_{A, \ell}\left(\left(\frac{K(A[\ell])/K}{\fp} \right)\right)$ is contained in $\cal{C}^{'CM}(\ell)$ and the reverse also holds for degree 1 primes. Thus 
\begin{align}
\pi_{A, \text{split}}(x) & \ll_K  x^{\frac{1}{2}}+ \pi_{\cal{C}^{'CM}(\ell)}(x, K(A[\ell])/K)  \nonumber \\
& =x^{\frac{1}{2}}+\pi_{\cal{C}^{CM}(\ell)}(x, K(A[\ell])^{\Lambda(\ell)}/K). \label{estimate-1-CM} 
\end{align}

We now give the proof of Theorem \ref{main-thm-CM} based on the observations above.
\begin{proof}
 Let $\ell$ be a prime for which $\ell\nmid N_A$ and splits in $E$. Since  Artin's Holomorphic Conjecture holds for the abelian extension $ K(A[\ell])^{\Lambda(\ell)}/K$, we invoke Theorem \ref{effective-CDT} (ii) with $K(A[\ell])^{\Lambda(\ell)}/K$ and $G=\T(\ell)$, (\ref{estimate-1-CM}), Proposition \ref{countings-CM}, Lemma \ref{hensel}, and the criterion of N\'{e}ron-Ogg-Shafarevich for abelian varieties to get
\begin{align*}
\pi_{A, \text{split}}(x) &\ll_K x^{\frac{1}{2}}+ \frac{|\cal{C}^{CM}(\ell)|}{|\T(\ell)/\Lambda(\ell)|}\Li(x) +|\cal{C}^{CM}(\ell)|^{\frac{1}{2}} x^{\frac{1}{2}}\left( \log (M(K(A[\ell])^{\Lambda(\ell)}/K) x)\right)
\\
& \ll_K \frac{x}{\ell\log x}+\ell^{\frac{1}{2}}x^{\frac{1}{2}}\log(\ell N_A x).
\end{align*}
Taking $\ell=\ell(x)\asymp_{K, E} \frac{x^{\frac{1}{3}}}{(\log x)^{\frac{4}{3}}}$, we obtain that for all sufficiently large $x$, 
\[
\pi_{A, \text{split}}(x) \ll_{A, K, E} x^{\frac{2}{3}}(\log x)^{\frac{1}{3}}.
\]
\end{proof}

\medskip
\subsection{Proof of  (\ref{cor-1}) and Corollary \ref{main-cor}}

\begin{proof}[Proof of   (\ref{cor-1})]
Recall from  \cite[Theorem 4.3, p. 332]{MaNa02} that each absolutely simple abelian surface  over $\F_q$ is $\F_q$-isogenous to the Jacobian of a smooth projective curve of genus two. Thus, for
any prime $\fp\in \Sigma_K$ such that   $\fp\nmid N_A$ and $N(\fp)\leq x$, if $\overline{A}_{\fp}$ is not isogenous to the Jacobian of a smooth projective curve, $\fp$ is also counted by $ \pi'_{A, \text{split}}(x)$. So the inequality (\ref{cor-1}) holds. Moreover, if $\End_{\overline{K}}(A)\simeq \Z$, then  by the remark after equation (\ref{abs-simple}), we get the same bound as in Theorem \ref{main-thm} under each conjecture. 
\end{proof}

\begin{proof}[Proof of Corollary \ref{main-cor}]
 For a prime $\fp\in \Sigma_K$ such that $\fp\nmid N_A$ and  $N(\fp) \leq x$, the extremal value  for $|\overline{A}_{\fp}(\F_{\fp})|$ is $(q+1\pm \lfloor 2\sqrt{q}\rfloor)^2$, where $q=|\F_{\fp}|$. Let $\fp$ be an extremal prime of $A$. By \cite[Corollary 2.2, p. 205 and Corollary 2.14, p. 213]{AuHaLa13},  the characteristic polynomial of $\overline{A}_{\fp}$  is   
\[
 P_{A, \fp}(X)=(X^2\pm \lfloor 2\sqrt{q}\rfloor X+q)^2=X^4\pm 2\lfloor 2\sqrt{q}\rfloor X^3+(2q +\lfloor 2\sqrt{q}\rfloor^2) X^2\pm 2q\lfloor 2\sqrt{q}\rfloor X+q^2.
\] 
In particular, $a_{1, \fp}=\pm 2\lfloor 2\sqrt{q}\rfloor\notin\{0, \pm \sqrt{q}, \pm \sqrt{5q}, \pm \sqrt{2q}, \pm 2\sqrt{q}\}$, $a_{2, \fp}=2q +\lfloor 2\sqrt{q}\rfloor^2 \notin \{0, \pm q, \pm 2q, 3q\}$, and 
$\Delta_{A, \fp}=0$ is a square. Combining these results  with  \cite[Theorem 2.9, p. 324]{MaNa02}, we conclude that $\overline{A}_{\fp}$ is not  simple. Thus, we obtain that for all sufficiently large $x$, 
\begin{align*}
& \#\{\fp\in \Sigma_K: N(\fp)\leq x, \fp\nmid N_A, |\overline{A}_{\fp}(\F_{\fp})|=(q+1\pm \lfloor 2\sqrt{q}\rfloor)^2, q=|\F_{\fp}|\}  \ll \pi_{A, \text{split}}(x).
\end{align*}
The conclusion then follows easily.
\end{proof}
\bigskip
\section{Final remarks}
\subsection{Endomorphism rings of abelian surfaces of type II and type IV}
We offer several remarks regarding absolutely simple abelian surfaces of type II and type IV in Albert's classification. 

Let $A/K$ be an absolutely simple abelian surface of type II. In this case, $\End_{\overline{K}}(A)\otimes_{\Z} \Q$ is isomorphic to an indefinite quaternion algebra over $K$, then by the Murty-Patankar Conjecture,  the density of primes $\fp$ in $K$ such that  $\overline{A}_{\fp}$ splits is not zero.  Indeed, by  \cite[Theorem B]{Ac09}, we have  $\pi_{A, \text{split}}(x)\sim \frac{x}{\log x}$. Moreover, there is a positive density set of primes $\fp$ for which $\overline{A}_{\fp}$ is isogenous over $\F_\fp$ to the square of an elliptic curve.

Now shifting our focus to absolutely simple abelian surfaces of type IV,  we have $\End_{\overline{K}}(A)\otimes_{\Z} \Q$ is isomorphic to a quartic CM field $F$. In particular,  there is a zero density of primes $\fp$ such that $\overline{A}_{\fp}$ splits since $\End_{\overline{K}}(A)$ is commutative. If $F$ is  a primitive quartic CM field, then  $\End_{\overline{K}}(A)$ is not isomorphic to the maximal order of $F$. Because otherwise, by  \cite[Theorem 1, p. 36, Theorem 2, p. 37]{Go97} and the Chebotarev Density Theorem, there is a positive density set of primes $\fp$ such that $\overline{A}_{\fp}$ is not simple.

Again, let $A/K$ be an absolutely simple abelian surface of type IV. Let $\fp\nmid N_A$ be a prime of degree 1. We can derive refined results for the density of $p$-ranks of the reduction $\overline{A}_{\fp}$. Specifically, it is known that the density of primes $\fp$ for which the $p$-rank of  $\overline{A}_{\fp}$ is 2 (supersingular case) equals 0 (e.g., see \cite[Lemma 2, p. 409]{Xi96}). Assume  $F/\Q$ is Galois. By \cite[case (c) in the proof of Theorem 3.7 (ii), p. 125]{Go98}, the density of  primes $\fp$ for which the $p$-rank of  $\overline{A}_{\fp}$ is 1 (intermediate case) equals $\frac{1}{4}$. Invoking this result with \cite[p. 567, Theorem 2]{Sa16}, we conclude that the density of  primes $\fp$ for which the $p$-rank of  $\overline{A}_{\fp}$ is 0 (ordinary case) equals $\frac{1}{[K':K]} = \frac{3}{4}$, where $K'$ is the smallest field that all the endomorphisms of $A$ are defined. But this is absurd as $[K':K]$ is an integer. This implies that $F/\Q$ is not Galois.  

\subsection{Effectiveness in computing endomorphism rings}
Let $A$ be an abelian surface defined over a number field $K$. The result \cite[Proposition 2.18, p. 11]{Lo19} indicates an effective computable constant $C_{A, K}'$ depending on $A$ and $K$ such that if the good reduction $\overline{A}_{\mathfrak{p}}$ splits for any prime $\mathfrak{p}$ in $K$ with norm bound $C_{A, K}'$,  then either $A$ splits  over $K$ or $\End_{\overline{K}}(A)$ is an order in a quaternion algebra. 
An interesting question is to write $C_{A, K}'$ explicitly regarding invariants of $A$ and $K$.

\subsection{The Conjecture (\ref{AH-conj})}
It would be interesting to investigate both theoretically and statistically how explicitly the constant $C_{A/K}$ in Conjecture (\ref{AH-conj})  depends on $A/K$. Furthermore, it is also interesting to produce a conjectural formula of $\pi_{A, \text{split}}(x)$, with explicit constant in the main term,  for a principally polarized abelian surface $A$ (such as the Jacobian of a genus 2 curve)  that has real or complex multiplication. 




{\small{

}}

\end{document}